\documentclass[12pt]{scrartcl}

\usepackage[]{amscd,amsmath, amssymb,amsfonts,amsthm,bm,mathtools,braket,color,enumitem}

\usepackage{color, xcolor, colortbl}
\usepackage{cite}
\usepackage{hhline}
\usepackage{url}
\usepackage{here}
\usepackage{tikz-cd}
\usepackage{tikz}
\usepackage{hyperref}
\hypersetup{
  colorlinks   = true, 
  urlcolor     = blue, 
  linkcolor    = blue, 
  citecolor   = red 
}
\usetikzlibrary{decorations}
\usetikzlibrary{arrows}
\tikzstyle{v} = [circle, draw, inner sep=2pt, minimum size=3pt, fill=black]

\usepackage{cleveref}
\usepackage{autonum}
\usepackage{comment}

\theoremstyle{plain}
\newtheorem{theorem}{Theorem}[section]
\newtheorem{lemma}[theorem]{Lemma}
\newtheorem{proposition}[theorem]{Proposition}
\newtheorem{corollary}[theorem]{Corollary}

\theoremstyle{definition}
\newtheorem{definition}[theorem]{Definition}

\newtheorem{example}[theorem]{Example}

\newtheorem{remark}[theorem]{Remark}

\newcommand{\DELTA}{\mathit{\Delta}}
\newcommand{\LAMBDA}{\mathit{\Lambda}}
\newcommand{\GAMMA}{\mathit{\Gamma}}
\newcommand{\PSI}{\mathit{\Psi}}

\DeclareMathOperator{\Char}{char}

\DeclareMathOperator{\Der}{Der}

\newcommand{\odd}{\textup{odd}}
\newcommand{\even}{\textup{even}}
\newcommand{\bal}{\textup{bal}}

\newcommand{\A}{\mathcal{A}}
\newcommand{\Amu}{\mathcal{A},\mu}

\numberwithin{equation}{section}

\begin{document}
\title{Exponents of $2$-multiarrangements of three lines over fields of positive characteristic}
\author{
Shuhei Tsujie
\thanks{Department of Mathematics, Hokkaido University of Education, Asahikawa, Hokkaido 070-8621, Japan. 
E-mail:tsujie.shuhei@a.hokkyodai.ac.jp}
\and
Ryo Uchiumi
\thanks{
Department of Mathematics, Graduate School of Science, The University of Osaka, Toyonaka, Osaka 560-0043, Japan. 
E-mail:uchiumi.ryou.1xu@ecs.osaka-u.ac.jp}
}
\date{}
\maketitle
\begin{abstract}
Wakamiko determined bases and the exponents for multiarrangements of three lines over a field of characteristic zero.
In this paper, we study the exponents of multiarrangements of three lines for the case of positive characteristic. 
We provide an effective algorithm for computing the exponents.
Furthermore, we prove that the multiplicity lattice has plenty of symmetries. 
We also discuss the validity of bases constructed from binomial expansions in the case of positive characteristic.
\end{abstract}



\tableofcontents


\section{Introduction}
Let $V$ be an $\ell$-dimensional vector space over a field $\mathbb{F}$.
Let $\{x_1,\ldots,x_\ell\}$ be a basis for $V^\ast$ and let $S = \mathbb{F}[x_1,\ldots,x_\ell]$ be the polynomial ring.
Define $\Der_S = \bigoplus_{i=1}^{\ell} S\partial_{x_i}$ as the module of polynomial vector fields, where $\partial_{x_i} = \frac{\partial}{\partial{x_i}}$.

A finite collection $\A = \{H_1,\ldots,H_n\}$ of linear hyperplanes in $V$ is called a (central) \textbf{hyperplane arrangement}. 
The pair $(\A,\mu)$ consisting of $\A$ and a map $\mu : \A \longrightarrow  \mathbb{Z}_{\geq 0}$ is called a $\ell$-\textbf{multiarrangement}.
We say that $\mu$ is \textbf{balanced} if $|\mu|/2 \geq \max_{H\in \mathcal{A}}\mu(H)$, where $|\mu| \coloneqq \sum_{H \in \A}\mu(H)$. 
Otherwise, we say that $\mu$ is \textbf{unbalanced}. 
For each $H \in \A$, choose a linear form $\alpha_H \in V^\ast$ such that $H = \ker{\alpha_H}$.
The polynomial 
\begin{align}
Q(\A,\mu) \coloneqq \prod_{H \in \A}\alpha_H^{\mu(H)}
\end{align}
is called the \textbf{defining polynomial} of the multiarrangement $(\A,\mu)$.

Define the \textbf{module of logarithmic vector fields} $D(\A,\mu)$ by 
\begin{align}
    D(\A,\mu) \coloneqq \Set{\theta \in \Der_S | \theta(\alpha_H) \in \alpha_H^{\mu(H)}S \quad \text{for any $H \in \A$}}.
\end{align}
The multiarrangement $(\A,\mu)$ is said to be \textbf{free with exponents} $\exp(\A,\mu) = (d_1,\ldots,d_\ell)$ if $D(\A,\mu)$ is a free $S$-module with a basis $\{\theta_1,\ldots,\theta_\ell\} \subseteq D(\A,\mu)$ of the form $\theta_i = \sum_{j=1}^\ell f_{ij}\partial_{x_j}$, where each non-zero $f_{ij}$ is a homogeneous polynomial of $\deg{f_{ij}} = d_i$.
We define $\deg{\theta_i} = d_i$.

Freeness is one of central topics of study of arrangements. 
In particular, it is still an open problem whether the freeness of a hyperplane arrangement over a fixed field is determined solely by the structure of its intersection lattice (Terao conjecture \cite[Problem 1]{terao1983exponents-spc}).
Freeness of multiarrangements plays an important role for understanding freenss of arrangements \cite{yoshinaga2004characterization-im, yoshinaga2005freeness-botlms, abe2013free-mz}.

For elements $\theta_1,\ldots,\theta_\ell \in D(\A,\mu)$ of the form $\theta_i = \sum_{j=1}^\ell f_{ij}\partial_{x_j}$, define the matrix $M(\theta_1,\ldots,\theta_\ell)$ by
\begin{align}
    M(\theta_1,\ldots,\theta_\ell) \coloneqq \Bigl(f_{ij}\Bigr)_{ij} = \begin{pmatrix}
f_{11} & \cdots & f_{1\ell}\\
\vdots & \ddots & \vdots\\
f_{\ell 1} & \cdots & f_{\ell\ell}
\end{pmatrix}.
\end{align}

Using this matrix, we have a useful criterion for freeness proved by K. Saito for simple arrangements \cite[Theorem 1.8]{saito1980theory-jotfosuotsim} and by Ziegler for multiarrangements \cite[Theorem and Definition 8]{Zi}. 

\begin{theorem}[Saito's criterion] \label{Saito-Ziegler}
Let $(\A,\mu)$ be a multiarrangement.
For $\theta_1,\ldots,\theta_\ell \in D(\A,\mu)$, the following are equivalent:
\begin{itemize}
\item $\{\theta_1,\ldots,\theta_\ell\}$ is a basis for $D(\A,\mu)$;
\item $\theta_1,\ldots,\theta_\ell$ are linearly independent over $S$ and $\sum_{i=1}^\ell \deg{\theta_i} = \sum_{H \in \A}\mu(H)$;
\item $\det{M(\theta_1,\ldots,\theta_\ell)} \doteq Q(\A,\mu)$,
\end{itemize}
where $A \doteq B$ means that there exists $c \in \mathbb{F}^\times$ such that $A = cB$.
\end{theorem}

In general, it is not easy to determine whether a multiarrangement is free or not. 
Every $2$-multiarrangement is known to be free {\cite[Corollary 7]{Zi}}.
However, it is also not easy to compute a basis and exponents. 

The first non-trivial case is the $2$-multiarrangement consisting of three lines. 
By a suitable linear transformation, we can assume that 
\begin{align}
\A = \{H_1,H_2,H_3\},\qquad H_1 \coloneqq \ker{x},\quad H_2 \coloneqq \ker{y},\quad H_3 \coloneqq \ker{(x+y)}, \label{A}
\end{align}
where $\{x,y\}$ denotes a basis for $V^{\ast}$. 
Let $\mu_{i} \coloneqq \mu(H_{i})$ for $i \in \{1,2,3\}$. 
Note that $\mu_{1},\mu_{2}$ and, $\mu_{3}$ have the same role by symmetry of $H_1$, $H_2$, and $H_3$.

When $\Char{\mathbb{F}} = 0$, Wakamiko gave a basis for $D(\Amu)$ and the exponents explicitly as follows.

\begin{theorem}[{Wakamiko \cite[Theorem 1.5]{W}}]\label{Wakamiko}
Suppose that $\Char{\mathbb{F}} = 0$ and $\mu_{3} \geq \max\{\mu_{1},\mu_{2}\}$. 
Then the following statements hold: 
\begin{enumerate}[label=(\arabic*)]
    \item\label{Wakamiko binomial type} If $\mu_1 + \mu_2 \leq \mu_3 + 1 $, then $\exp(\Amu) = (\mu_1+\mu_2,\, \mu_3)$ and 
\begin{align}
\left\{
\sum_{j = \mu_1}^{\mu_3}\binom{\mu_3}{j}x^jy^{\mu_3-j}\partial_x + \sum_{j=0}^{\mu_1-1}\binom{\mu_3}{j}x^jy^{\mu_3-j}\partial_y,\quad 
x^{\mu_1}y^{\mu_2}(\partial_y - \partial_x)
\right\}
\end{align}
is a basis for $D(\Amu)$.
    \item If $\mu_1 + \mu_2 > \mu_3 + 1 $, then 
    \begin{align}
    	\exp(\Amu) = \left(\left\lfloor \frac{|\mu|}{2}\right\rfloor,\  \left\lceil \frac{|\mu|}{2} \right\rceil \right),
    \end{align}
    where $\lfloor x \rfloor = \max{\Set{n \in \mathbb{Z} | n \leq x}}$, and $\lceil x \rceil = \min{\Set{n \in \mathbb{Z} | x \leq n}}$.
    In this case, a basis for $D(\Amu)$ can be expressed using generalized binomial coefficients (see \cite{W} for the description).
\end{enumerate}
\end{theorem}

\begin{remark}
Feigin, Wang, and Yoshinaga \cite[Theorem 2.3]{FWY} constructed bases for $D(\Amu)$ with integral expressions for balanced cases. 
\end{remark}

Let $\DELTA(\mu)$ denote the absolute value of the difference between the exponents of $(\A,\mu)$.
By \Cref{Wakamiko}, the exponents can be divided into unbalanced and balanced cases as follows. 
\begin{corollary}
Suppose that $\Char{\mathbb{F}} = 0$. 
    \begin{enumerate}[label=(\arabic*)]
        \item If $\mu$ is unbalanced, then 
        \begin{align}
            \exp(\Amu) = \bigl(|\mu| - \max{(\mu)},\, \max{(\mu)}\bigr),\quad \DELTA(\mu) = 2 \max(\mu) - |\mu|.
        \end{align}
        \item If $\mu$ is balanced, then
        \begin{align}
            \exp(\Amu) = \left(\left\lfloor \frac{|\mu|}{2}\right\rfloor,\  \left\lceil \frac{|\mu|}{2} \right\rceil \right),\qquad
        \DELTA(\mu) = \begin{cases*}
            1 & if $|\mu|$ is odd;\\
            0 & if $|\mu|$ is even.
        \end{cases*}
        \end{align}
    \end{enumerate}
\end{corollary}

In the case of positive characteristic, the behavior of the exponents is different from the case of characteristic zero (see \Cref{tab:intro exponents table} for example). 
\begin{table}[t]
    \centering
    \caption{$\DELTA(\mu)$ for $\mu = (\mu_1,\mu_2,16)$ when $p=3$}
    \includegraphics[width=\linewidth]{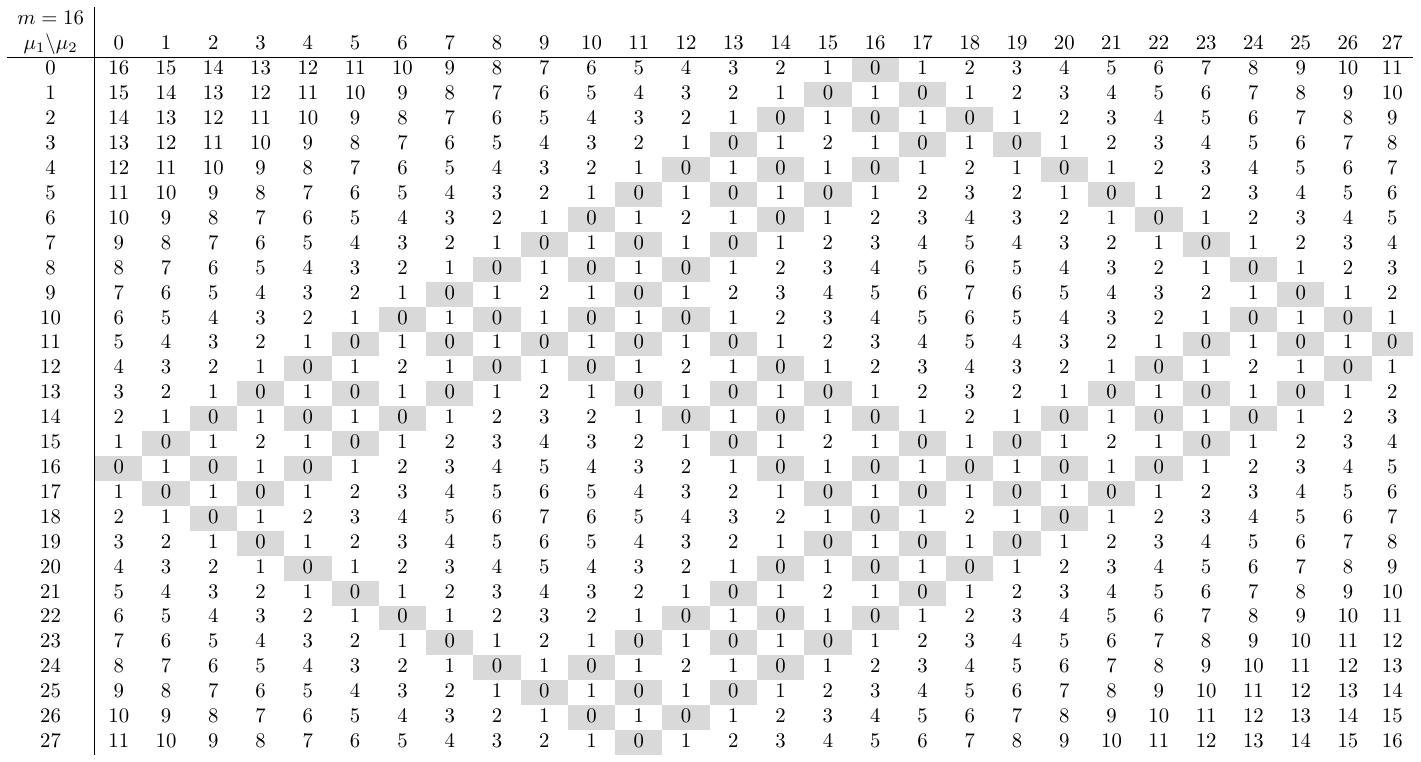}
    \label{tab:intro exponents table}
\end{table}
When $\Char{\mathbb{F}} = 0$, we can take a basis $\{\theta_\mu,\theta_\mu'\}$ for $D(\Amu)$ with integer coefficients by Theorem \ref{Wakamiko}. 
As shown below, the reduction of $\{\theta_\mu,\theta_\mu'\}$ modulo $p$ may not be a basis for positive characteristic. 

\begin{example} \label{eg334}
    Suppose that $\mu = (\mu_1, \mu_2, \mu_3) = (3,3,4)$.
    When $\Char{\mathbb{F}} = 0$, 
    \begin{align}
	   \theta_\mu &\coloneqq (x^5+4x^4y + 6x^3y^2) \partial_x + (4x^2y^3 + xy^4) \partial_y,\\
       \theta'_\mu &\coloneqq (x^5 + 5x^4y + 10x^3y^2) \partial_x + (10x^2y^3 + 5xy^4 + y^5) \partial_y
    \end{align}
    form a basis for $D(\Amu)$ (See \cite[Example 4.1]{W}) since 
    \begin{align}
	   \det{M(\theta_\mu,\theta'_\mu)} = 6x^3y^3(x+y)^4.
    \end{align}
    Therefore one can see that $\exp(\Amu) = (5,5)$ when $\Char{\mathbb{F}} \neq 2,3$. 
    However, when $\Char{\mathbb{F}} = 2$ or $3$, $\{\theta_\mu,\theta'_\mu\}$ is not a basis for $D(\Amu)$ since $\det{M(\theta_\mu,\theta_\mu')} = 0$.
    In fact, we will see later (\Cref{eg:334_ch23}) that $\exp(\Amu) = (4,6)$ when $\Char{\mathbb{F}} = 2$ or $3$.
\end{example}

The main purpose of this paper is to study $\exp(\Amu)$ for the case $\Char{\mathbb{F}} = p > 0$. 
The organization of this article is as follows.

One of the most significant differences between fields of positive characteristic and characteristic zero is the existence of the Frobenius endomorphism in positive characteristic.
In \Cref{section:Preliminaries}, we first prove that the Frobenius endomorphism preserves a basis for the module of logarithmic vector fields (\Cref{Period1}), followed by a review of the theory of multiplicity lattice of $2$-multiarrangements. 

If we remove the multiplicities $\mu$ such that $\DELTA(\mu)=0$ from the multiplicity lattice (the set of multiplicities), then the remaining part decomposes into some connected components. 
In order to determine the exponents, it suffices to know that the centers of the connected components (see \Cref{tab:intro exponents table} for example). 
In \Cref{section3}, we characterize the centers of the connected components (\Cref{thmcc}) and determine all centers (\Cref{Z(1)} and \Cref{c210}). 
Using these results, we provide an effective algorithm for computing the exponents (\Cref{main thm:expnonents}). 

As shown in \Cref{tab:intro exponents table}, the multiplicity lattice has a lot of symmetries. 
In \Cref{section:plane}, we prove the periodicity (\Cref{Period2}) and symmetry (\Cref{veebasis}) by constructing correspondences between bases.

The bases constructed by Wakamiko in \Cref{Wakamiko} \ref{Wakamiko binomial type} are valid for positive characteristic. 
Since binomial coefficients may vanish in positive characteristic, the range in which the bases are valid becomes broader than the case of characteristic zero. 
In \Cref{section:gammam}, we explicitly determine this range (\Cref{main thm:gammam}).

\section{Preliminaries}\label{section:Preliminaries}

\subsection{Action of the Frobenius endomorphism}
Let $q = p^{d}$ be a prime power. 
In this subsection, we assume that $\mathcal{A}$ is an $\ell$-arrangement over the finite field $\mathbb{F}_{q}$. 
Let $\phi_{q} \colon S \longrightarrow S$ denote the $q$-th power Frobenius endomorphism. 
Note that 
\begin{align}
\phi_{q}(f(x_1, \dots, x_\ell)) = f(x_1, \dots, x_\ell)^{q} = f(x_1^q, \dots, x_\ell^q)
\end{align}
for any $f \in S$. 

The map $\phi_{q}$ is injective and induces the endomorphism on $\Der_S$ defined by 
\begin{align}
\phi_{q}\left( \sum_{i=1}^{\ell} f_{i} \partial_{x_i}\right) \coloneqq \sum_{i=1}^{\ell}\phi_{q}(f_{i})\partial_{x_i}. 
\end{align}
Then we have $(\phi_{q}(\theta))(\alpha) = \phi_{q}(\theta(\alpha))$ for any $\theta = \sum_{i=1}^{\ell}f_{i}\partial_{x_{i}} \in \Der_{S}$ and a linear form $\alpha = \sum_{i=1}^{\ell}c_{i}x_{i} \ (c_{i} \in \mathbb{F}_{q})$ since 
\begin{align}
\phi_{q}(\theta(\alpha)) &= \left(\sum_{i=1}^{\ell}f_{i}\partial_{x_{i}}\sum_{i=1}^{\ell}c_{i}x_{i}\right)^{q}
= \left(\sum_{i=1}^{\ell}c_{i}f_{i}\right)^{q} \\
&= \sum_{i=1}^{\ell}c_{i}f_{i}^{q}
=\sum_{i=1}^{\ell}f_{i}^{q}\partial_{x_{i}}\sum_{i=1}^{\ell}c_{i}x_{i}
=(\phi_{q}(\theta))(\alpha). 
\end{align}

For a multiplicity $\mu$ on $\mathcal{A}$, let $q\mu$ denote the multiplicity defined by $(q\mu)(H) = q\cdot \mu(H)$ for any $H \in \A$. 
The Frobenius endomorphism raises the contact order of vector fields
to each hyperplane. 

\begin{lemma}\label{Frob raises contact order}
Let $\mu$ be a multiplicity on $\mathcal{A}$. 
Then 
\begin{align}
\phi_{q}(D(\mathcal{A},\mu)) = D(\mathcal{A},q\mu) \cap \operatorname{Im}\phi_{q}. 
\end{align}
\end{lemma}

\begin{proof}
First, we suppose $\theta \in D(\A,\mu)$.
Then for any $H \in \mathcal{A}$, there exists $f \in S$ such that $\theta(\alpha_{H}) = \alpha_{H}^{\mu(H)}f$. 
Hence
\begin{align}
\phi_{q}(\theta)(\alpha_{H})
=\phi_{q}(\theta(\alpha_H)) 
= \phi_{q}(\alpha_{H}^{\mu(H)}f)
= \alpha_{H}^{q\mu(H)}f^{q}
\in \alpha_{H}^{q\mu(H)}S. 
\end{align}
Therefore $\phi_{q}(\theta)\in D(\mathcal{A},q\mu)$ and $\phi_{q}(D(\mathcal{A},\mu)) \subseteq D(\mathcal{A},q\mu) \cap \operatorname{Im}\phi_{q}$. 

Next, suppose that $\theta \in \Der_{S}$ and $\phi_{q}(\theta) \in D(\mathcal{A},q\mu)$. 
Then for any $H \in \mathcal{A}$, there exists $f \in S$ such that $\phi_{q}(\theta)(\alpha_{H}) = \alpha_{H}^{q\mu(H)}f$. 
Furthermore, $\phi_{q}(\theta)(\alpha_{H}) = \phi_{q}(\theta(\alpha_{H})) = (\theta(\alpha_{H}))^{q}$. 
Hence $(\theta(\alpha_{H}))^{q} = \alpha_{H}^{q\mu(H)}f$. 
Therefore there exists $g \in S$ such that $f = g^{q}$. 
Since the map $\phi_{q} \colon S \longrightarrow S$ is injective, we have $\theta(\alpha_{H}) = \alpha_{H}^{\mu(H)}g \in \alpha_{H}^{\mu(H)}S$. 
Thus $\theta \in D(\mathcal{A},\mu)$ and $\phi_q(D(\mathcal{A},\mu)) \supseteq D(\mathcal{A},q\mu) \cap \operatorname{Im}\phi_{q}$. 
\end{proof}

In \Cref{Frob raises contact order}, if $D(\mathcal{A},\mu)$ is free, then we can show that $D(\mathcal{A},q\mu)$ is also free. 
In fact, the Frobenius endomorphism $\phi_{q}$ preserves a basis as follows.

\begin{theorem}\label{Period1}
Let $\mu$ be a multiplicity on $\mathcal{A}$ and $\theta_{1}, \dots, \theta_{\ell} \in \Der_{S}$. 
    Then the following are equivalent:
    \begin{enumerate}[label=(\arabic*)]
        \item $\{\theta_1,\ldots,\theta_\ell\}$ is a basis for $D(\Amu)$; 
        \item $\{\phi_q(\theta_1),\ldots, \phi_q(\theta_\ell)\}$ is a basis for $D(\A,q\mu)$. 
    \end{enumerate}
\end{theorem}
\begin{proof}
    Since 
    $
    \det M\bigl(\phi_q(\theta_1),\ldots,\phi_q(\theta_\ell)\bigr) = \phi_q\bigl( \det M(\theta_1,\ldots,\theta_\ell)\bigr)
    $
    and $\phi_q$ is injective, we see that $\det M(\theta_1,\ldots,\theta_\ell) \doteq Q(\A,\mu)$ if and only if $\det M\bigl(\phi_q(\theta_1),\ldots,\phi_q(\theta_\ell)\bigr) \doteq Q(\A,q\mu)$.
    By \Cref{Frob raises contact order} and Saito's criterion (\Cref{Saito-Ziegler}), we can conclude that the assertion is true. 
\end{proof}

\subsection{The multiplicity lattice}

Suppose that $\mathcal{A} = \{H_{1}, \dots, H_{n} \}$ is a $2$-arrangement over an arbitrary field $\mathbb{F}$. 
Given a multiplicity $\mu : \mathcal{A} \to \mathbb{Z}_{\geq 0}$, let $\mu_{i} \coloneqq \mu(H_{i})$ for each $i \in \{1, \dots, n\}$. 
We identify the multiplicity $\mu$ and the element $(\mu_{1}, \dots, \mu_{n})$ in $\LAMBDA \coloneqq (\mathbb{Z}_{\geq 0 })^{n}$ and call $\LAMBDA$ the \textbf{multiplicity lattice}. 
Recall $\mu$ is said to be \textbf{balanced} if $|\mu|/2 \geq \mu_{i}$ for each $i \in \{1, \dots, n\}$, where $|z| = |z_{1}| + \dots + |z_{n}|$ for $z=(z_1,\dots,z_n) \in \mathbb{Z}^{n}$. 

Let $\mu \in \LAMBDA$. 
Since the 2-multiarrangement $(\mathcal{A},\mu)$ is free, we can take a homogeneous basis $\{\theta_\mu,\theta_\mu'\}$ of $D(\Amu)$ such that $\deg{\theta_\mu} \leq \deg{\theta_\mu'}$.
We call $\theta_\mu$ the \textbf{lower degree basis}.
Note that the lower degree basis is uniquely determined up to scalar multiplication if $\deg{\theta_\mu} < \deg{\theta_\mu'}$. 

Define the function $\DELTA : \LAMBDA \longrightarrow \mathbb{Z}_{\geq 0}$ by 
\begin{align}
\DELTA(\mu) = \deg{\theta_\mu'} - \deg{\theta_\mu}. 
\end{align}
Since $\deg{\theta_\mu'} + \deg{\theta_\mu} = |\mu|$, if we know $\DELTA(\mu)$, then we can obtain the exponents $\exp(\mathcal{A},\mu) = (\deg{\theta_\mu},\, \deg{\theta_\mu'})$, and vice versa. 

Two elements $\mu,\nu \in \LAMBDA$ are said to be \textbf{adjacent} if $\mu-\nu = \pm \bm{e}_{i}$ for some $i \in \{1, \dots, n\}$, where $\bm{e}_{i}$ denotes the $i$-th standard unit vector.

\begin{lemma}[Abe--Numata {\cite[Lemma 4.2 and 4.3]{AN}}]\label{AN1}
    Let $\mu,\nu \in \LAMBDA$ be adjacent. 
    Then the following statements hold:
    \begin{enumerate}[label=(\arabic*)]
        \item $|\DELTA(\mu) - \DELTA(\nu)| = 1$.
        \item If $\nu(H) = \mu(H) + 1$ for $H \in \A$, then
        \begin{align}
            \theta_\nu = \begin{cases*}
                \alpha_H\theta_\mu & if \ $\DELTA(\nu) < \DELTA(\mu)$;\\
                \theta_\mu & if \ $\DELTA(\nu) > \DELTA(\mu)$.
            \end{cases*}
        \end{align}
    \end{enumerate}
\end{lemma}

Define the subsets $\LAMBDA^\bal, \LAMBDA_{>0}$, and $\LAMBDA_{>0}^\bal$ of $\LAMBDA$ by
\begin{align}
\LAMBDA^\bal \coloneqq \Set{\mu \in \LAMBDA | \mu \text{ is balanced}}, \quad
\LAMBDA_{>0} \coloneqq \Set{\mu \in \LAMBDA | \DELTA(\mu) > 0}, \quad
\LAMBDA_{>0}^\bal \coloneqq \LAMBDA^\bal \cap \LAMBDA_{>0}.
\end{align}
The adjacentness induces an equivalence relation on $\LAMBDA_{>0}$ and 
every equivalence class is called a \textbf{connected component} of $\LAMBDA_{>0}$. 

\begin{theorem}[Abe--Numata {\cite[Theorem 3.1 and Lemma 4.5]{AN}}]\label{CCThm:unbalanced}
The following statememts hold: 
\begin{enumerate}[label=(\arabic*)]
\item Let $C_{i} \coloneqq \Set{\mu \in \LAMBDA | \mu_{i} > \frac{1}{2}|\mu|}$ for each $H_{i} \in \mathcal{A}$. 
Then $C_{i}$ is a connected component of $\LAMBDA_{>0}$. 
\item $\LAMBDA_{>0}\setminus\LAMBDA_{>0}^{\bal} = \bigcup_{i=1}^{n}C_i$. 
\item Suppose that both $\mu$ and $\mu+\boldsymbol{e}_{i}$ belong to $C_i$. 
Then $\DELTA(\mu+\boldsymbol{e}_{i}) = \DELTA(\mu)+1$. 
In particular, $\DELTA(\mu) = 2\mu_i - |\mu|$ for any $\mu \in C_{i}$.
\end{enumerate}
\end{theorem}

For $\zeta \in \LAMBDA$ and $r \in \mathbb{Z}_{>0}$, define the \textbf{ball} $B(\zeta,r)$ with radius $r$ and center $\zeta$ by
\begin{align}
    B(\zeta,r) \coloneqq \Set{\mu \in \LAMBDA |{|\mu-\zeta| < r}}.
\end{align}

\begin{theorem}[Abe--Numata {\cite[Theorem 3.1, Theorem 3.2]{AN}}]\label{CCThm}
The following statements hold: 
\begin{enumerate}[label=(\arabic*)]
\item $\LAMBDA_{>0}^{\bal} = \bigcup_{C}C$, where $C$ runs over the all finite connected components of $\LAMBDA_{>0}$. 
\item Let $C$ be a connected component of $\LAMBDA^\bal_{>0}$.
Then there exists $\zeta \in C$ such that $C = B(\zeta,\DELTA(\zeta))$.
Moreover, for $\mu \in C$,
\begin{align}
    \DELTA(\mu) = \DELTA(\zeta) - |\mu-\zeta|.
\end{align}
In particular, the center of $C$ is characterized as a unique element $\zeta \in C$ such that $\DELTA(\zeta) > \DELTA(\mu)$ for any $\mu$ adjacent to $\zeta$. 
\end{enumerate}
\end{theorem}

\begin{corollary}[Abe--Numata {\cite[Corollary 3.3]{AN}}] \label{ANC33}
    Let $C$ be a connected component of $\LAMBDA_{>0}^\bal$ and $\zeta$ its center. 
    If $\mu \in \LAMBDA$ satisfies $|\mu - \zeta| < \DELTA(\zeta) + 2$, then 
    \begin{align}
        \DELTA(\mu) = \bigl|\DELTA(\zeta) - |\mu - \zeta|\bigr|.
    \end{align}
\end{corollary}

\begin{corollary}\label{c3}
    Let $\mu \in \LAMBDA_{>0}$ and $C$ the connected component containing $\mu$.
    Then $B(\mu,\DELTA(\mu)) \subseteq C$.
    In particular, if $\nu \in \LAMBDA_{>0}$ satisfies $|\nu - \mu| < \DELTA(\mu)$, then $\mu$ and $\nu$ belong to the same connected component. 
\end{corollary}
\begin{proof}
    Suppose that $\mu \in \LAMBDA_{>0} \setminus \LAMBDA_{>0}^\bal$.
    Then $C = C_i$ for some $H_i \in \A$ by \Cref{CCThm:unbalanced}.
    Let $\nu \in B(\mu,\DELTA(\mu))$.
    Then $h \coloneqq |\nu - \mu| < 2\mu_i - |\mu|$.
    If $\nu_i = \mu_i + t$ for $0 \leq t \leq h$, then $|\nu| \leq |\mu| + h$.
    Hence
    \begin{align}
        \nu_i - \dfrac{|\nu|}{2} \geq \mu_i + t - \dfrac{|\mu| + h}{2} > t \geq 0.
    \end{align}
    If $\nu_i = \mu_i - t$ for $0 < t \leq h$, then $|\nu| \leq |\mu| + h - 2t$.
    Hence
    \begin{align}
        \nu_i - \dfrac{|\nu|}{2} \geq \mu_i - t - \dfrac{|\mu| + h - 2t}{2} = \mu_i - \dfrac{|\mu|}{2} - \dfrac{h}{2} > 0.
    \end{align}
    Therefore we have $\nu \in C_i$.

    Suppose that $\mu \in \LAMBDA_{>0}^\bal$.
    Then there exists $\zeta \in \LAMBDA_{>0}^\bal$ such that $C = B(\zeta,\DELTA(\zeta))$. 
    Let $\nu \in B(\mu,\DELTA(\mu))$.
    Then
    \begin{align}
        |\nu-\zeta| \leq |\nu - \mu| + |\mu-\zeta| < \DELTA(\mu) + |\mu-\zeta| = \DELTA(\zeta)
    \end{align}
    by \Cref{CCThm}. 
    Hence $\nu \in C$.
\end{proof}

\begin{definition}
Let $Z$ denote the set consisting of the centers of the connected components of $\LAMBDA_{>0}^{\bal}$. 
\end{definition}
By \Cref{CCThm}, if we can determine the centers $\zeta \in Z$ and the radii $\DELTA(\zeta)$, then we can obtain $\DELTA(\mu)$ for any $\mu \in \LAMBDA^\bal_{>0}$.

Abe \cite[Theorem 1.6]{AA} showed that, when $\Char\mathbb{F} = 0$, if $\#\mathcal{A} \geq 3$ and $\mu$ is balanced, then $\DELTA(\mu) \leq \#\mathcal{A}-2$. 
However, when $\mathbb{F}$ is a finite field, $\DELTA(\mu)$ is not bounded because of the Frobenius endomorphism (See \Cref{Period1}). 
We list properties of the multiplicity lattice in the case $\mathbb{F}$ is finite as follows. 

\begin{proposition}\label{self-similarity}
Suppose that $\mathbb{F}$ is a finite field $\mathbb{F}_{q}$. 
Then the following statements hold: 
\begin{enumerate}[label=(\arabic*)]
\item\label{ss1} For any $\mu \in \LAMBDA$, we have $\DELTA(q\mu) = q \cdot \DELTA(\mu)$. 
\item\label{ss2} Let $\mu \in \LAMBDA_{>0}^{\bal}$. 
Then $\mu \in Z$ if and only if $q\mu \in Z$. 
\end{enumerate}
\end{proposition}
\begin{proof}
(1) Suppose that $\exp(\mathcal{A}, \mu) = (d_{1},d_{2})$. 
By \Cref{Period1}, we have $\exp(\mathcal{A}, q\mu) = (qd_{1},qd_{2})$. 
Therefore $\DELTA(q\mu) = |qd_1-qd_2| = q|d_1-d_2|=q\cdot \DELTA(\mu)$.

(2) First, suppose that $\mu \in Z$. 
Let $\nu$ be an arbitrary element adjacent to $\mu$. 
Then $\DELTA(\nu) = \DELTA(\mu)-1$ by \Cref{CCThm}. 
Hence 
\begin{align}
\DELTA(q\nu) = q \cdot \DELTA(\nu) = q \cdot \DELTA(\mu)-q = \DELTA(q\mu)-q. 
\end{align}
Since $|q\mu-q\nu| = q|\mu-\nu| = q$ and the values of the function $\DELTA$ of two adjacent elements differ exactly $1$ by \Cref{AN1}, the values of $\DELTA$ decrease along the path from $q\mu$ to $q\nu$. 
In particular, $\DELTA(q\mu)$ is greater than $\DELTA(\rho)$ for any $\rho$ adjacent to $q\mu$. 
Thus $q\mu \in Z$. 

Next, suppose that $q\mu \in Z$ and let $\nu$ be an element adjacent to $\mu$. 
Then 
\begin{align}
|q\mu-q\nu| = q|\mu-\nu| = q \leq q \cdot \DELTA(\mu) = \DELTA(q\mu). 
\end{align}
Therefore, by \Cref{ANC33}, 
\begin{align}
q\cdot \DELTA(\nu) = \DELTA(q\nu) = \bigl| \DELTA(q\mu) - |q\mu-q\nu| \bigr| = q(\DELTA(\mu)-1). 
\end{align}
Therefore $\DELTA(\nu)=\DELTA(\mu)-1$. 
Thus $\mu \in Z$. 
\end{proof}

\begin{remark}
\begin{figure}
    \centering
    \includegraphics[width=0.7\linewidth]{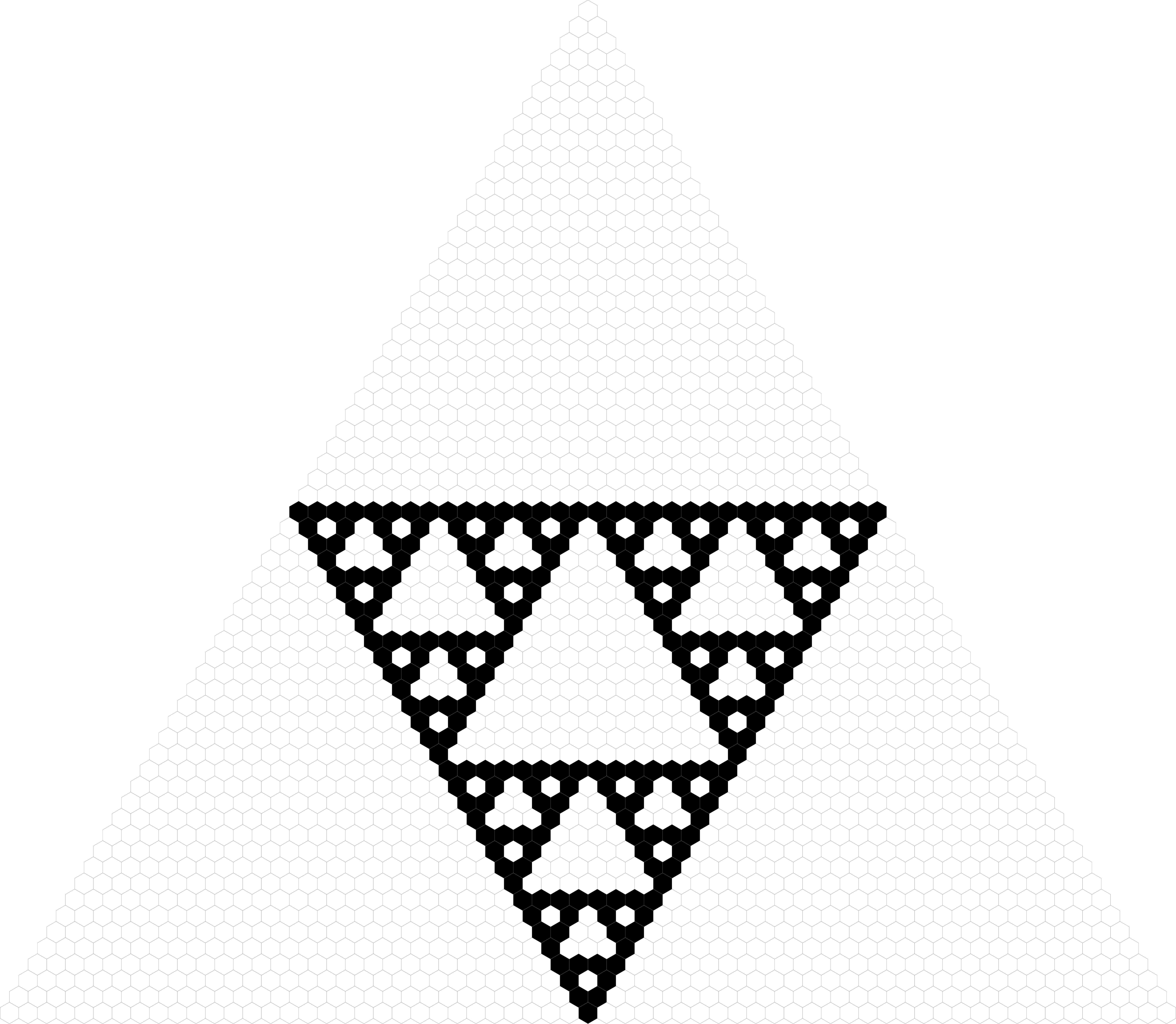}
    \caption{$p=2, \ |\mu| = 2 \cdot 2^5-2$}
    \label{fig:pascal_mod2}

    \ \vspace{3mm}

    \centering
    \includegraphics[width=0.7\linewidth]{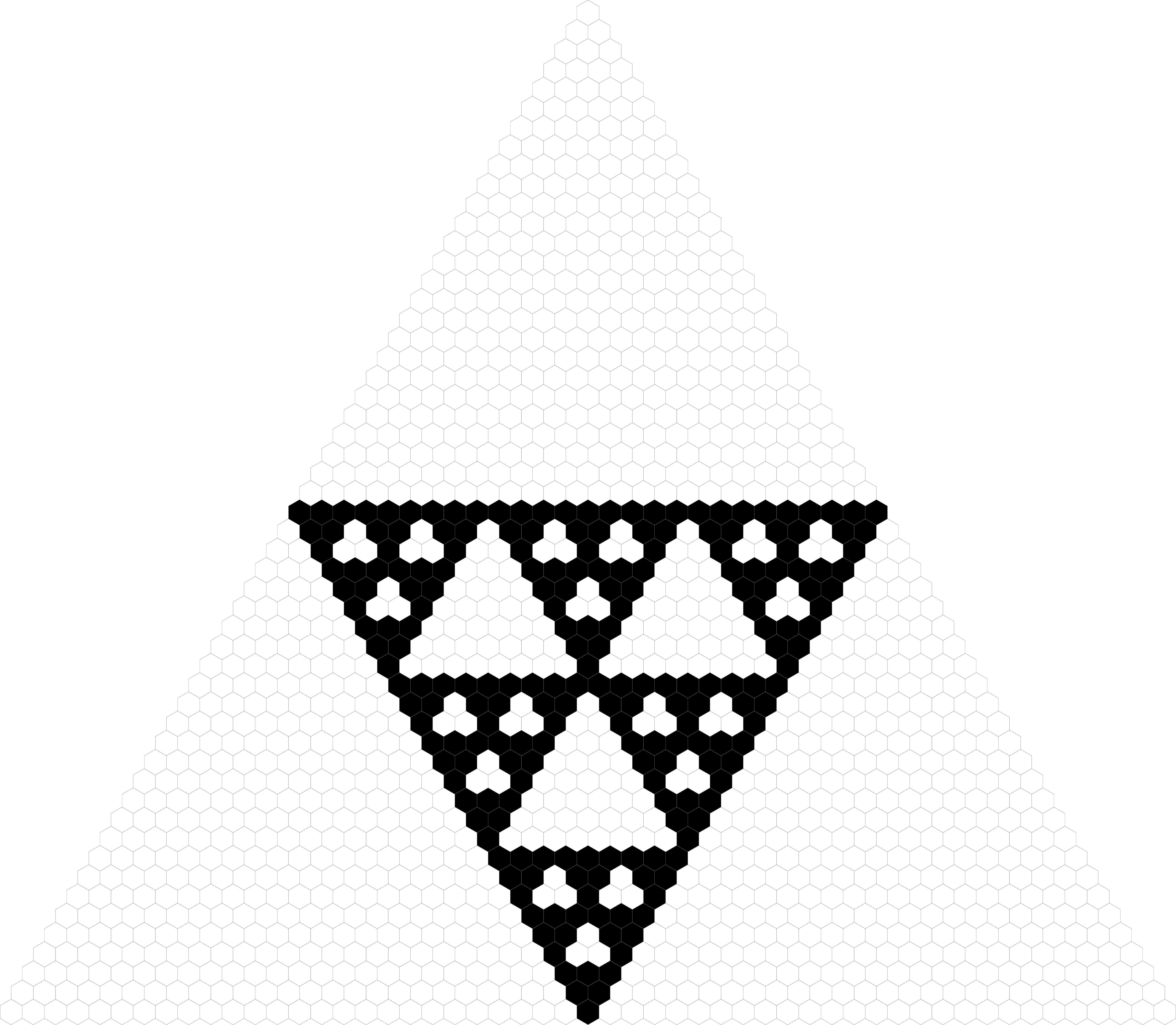}
    \caption{$p=3, \ |\mu| = 2 \cdot 3^3-2$}
    \label{fig:pascal_mod3}
\end{figure}
By \Cref{self-similarity}, the multiplicity lattice of an arrangement over a finite field has ``self-similarity". 
For example, let $\mathcal{A}$ be an arrangement consisting of three lines over $\mathbb{F}_{p}$ with $p$ prime. 
Then one will observe that the shape of the Pascal's triangle modulo $p$ appears on the plane defined by $|\mu|=2p^k-2$ in the multiplicity lattice. 
See \Cref{fig:pascal_mod2} and \ref{fig:pascal_mod3}, where a filled cell denotes a multiplicity $\mu$ with $\DELTA(\mu)=0$. 
\end{remark}

Abe \cite[Remark 3.2]{AA} also mentioned a behavior of the lower degree basis in the case of positive characteristic. 
The following proposition presents properties of the multiplicity lattice in positive characteristic.  
\begin{proposition}\label{lower degree is divisible by p}
Suppose that $\Char\mathbb{F} = p > 0$ and $\zeta \in Z$. 
If $\#\mathcal{A} \geq 3$ and $\DELTA(\zeta) > \#\mathcal{A}-2$, then the following statements hold: 
\begin{enumerate}[label=(\arabic*)]
\item\label{ldb1} Write the lower degree basis $\theta_{\zeta}$ as $\theta_{\zeta} = f\partial_{x} + g\partial_{y}$. 
Then $f,g \in \mathbb{F}[x^{p},y^{p}]$. 
In particular, $\deg \theta_{\zeta}$ is divisible by $p$. 
\item\label{ldb2} $\zeta \in p\LAMBDA_{>0}^\bal$.
In particular, $\DELTA(\zeta) \in p\mathbb{Z}$.
\end{enumerate}
\end{proposition}
\begin{proof}
\ref{ldb1}  This is proven by Abe \cite[Remark 3.2]{AA}. 

\ref{ldb2}  Let $H \in \mathcal{A}$. 
Without loss of generality, we can assume that $\alpha_{H} = x$. 
We will show $\zeta(H) \in p\mathbb{Z}$.
Define $k \coloneqq \min{\Set{k \in \mathbb{Z} | \zeta(H) \leq pk}}$.
By \ref{ldb1}, we have $\theta_\zeta(x) \in x^{pk}S$.

Assume that $\zeta(H) < pk$.
Let $\zeta'$ be the multiplicity adjacent to $\zeta$ satisfying $\zeta'(H) = \zeta(H) + 1$. 
Then $\DELTA(\zeta') < \DELTA(\zeta)$ by \Cref{CCThm}, and hence $\theta_{\zeta'} = x\theta_\zeta$ by \Cref{AN1}.
Since $\zeta(H)+1 \leq pk$, we have $\theta_\zeta \in D(\A,\zeta')$.
Since $\theta_{\zeta'}$ is the lower degree basis for $D(\A,\zeta')$, we have 
\begin{align}
    \deg\theta_\zeta \geq \deg\theta_{\zeta'} = \deg(x\theta_\zeta) = \deg\theta_\zeta + 1,
\end{align}
which is a contradiction.
Hence $\zeta(H) = pk \in p\mathbb{Z}$.

Let $\exp(\mathcal{A},\zeta) = (d_1,d_2)$ with $d_1 \leq d_2$. 
By \ref{ldb1}, $d_{1} \in p\mathbb{Z}$. 
Moreover, we have $d_1+d_2 = |\zeta| \in p\mathbb{Z}$. 
Thus $\DELTA(\zeta) = d_2-d_1 \in p\mathbb{Z}$. 
\end{proof}

\section{Exponents} \label{section3}
In the rest of this paper, let $(\mathcal{A}, \mu)$ be the $2$-multiarrangement over a finite field $\mathbb{F}_{p}$ with $p$ prime defined by 
\begin{align}
Q(\mathcal{A},\mu) = x^{\mu_1}y^{\mu_2}(x+y)^{\mu_{3}}. 
\end{align}

When $\mu$ is unbalanced, by \Cref{CCThm:unbalanced}, 
\begin{align}
\DELTA(\mu) = 2\max(\mu) - |\mu| \text{\quad and \quad} \exp(\mathcal{A},\mu) = \bigl( |\mu| - \max(\mu),\, \max(\mu) \bigr).
\end{align}
From now on, we will discuss balanced multiplicities. 

\subsection{Centers of connected components}
Recall 
\begin{align}
\LAMBDA^\bal = \Set{\mu \in \LAMBDA | \mu \text{ is balanced}}, \quad
\LAMBDA_{>0} = \Set{\mu \in \LAMBDA | \DELTA(\mu) > 0}, \quad
\LAMBDA_{>0}^\bal = \LAMBDA^\bal \cap \LAMBDA_{>0}.
\end{align}
Moreover, we will use the following notation. 
\begin{align}
\LAMBDA_\odd &\coloneqq \Set{\mu \in \LAMBDA | \text{$\DELTA(\mu)$ is odd}} = \Set{\mu \in \LAMBDA | \text{$|\mu|$ is odd}}, 
&\LAMBDA_\odd^\bal &\coloneqq \LAMBDA^\bal \cap \LAMBDA_\odd,\\
\LAMBDA_\even &\coloneqq \Set{\mu \in \LAMBDA | \text{$\DELTA(\mu)$ is even}} = \Set{\mu \in \LAMBDA | \text{$|\mu|$ is even}}, 
&\LAMBDA_\even^\bal &\coloneqq \LAMBDA^\bal \cap \LAMBDA_\even.
\end{align}

\begin{theorem}\label{thmcc}
	Every connected component of $\LAMBDA^\bal_{>0}$ is of the form $B(\zeta,p^k)$ for $k \in \mathbb{Z}_{\geq 0}$ and $\zeta \in p^k\LAMBDA^\bal_{\odd}$.
\end{theorem}
\begin{proof}
	Let $C$ be a connected component of $\LAMBDA_{>0}^\bal$.
    By \Cref{CCThm}, there exists $\zeta \in \LAMBDA_{>0}^\bal$ such that $C = B(\zeta,\DELTA(\zeta))$.
    Since $\#\mathcal{A} = 3$ and $\mathcal{A}$ is defined over the finite field $\mathbb{F}_{p}$, 
	\Cref{self-similarity} \ref{ss2} and \Cref{lower degree is divisible by p} \ref{ldb2} imply that there exists $\zeta' \in \LAMBDA^\bal_{>0}$ such that $\zeta = p^k\zeta'$ and $\DELTA(\zeta') = 1$.
	Hence $\DELTA(\zeta) = p^k\DELTA(\zeta') = p^k$ by \Cref{self-similarity} \ref{ss1}.
	Moreover, $\zeta' \in \LAMBDA_\odd^\bal$ since $\DELTA(\zeta')=1$ is odd.
    Thus $\zeta \in p^k\LAMBDA^\bal_{\odd}$. 
\end{proof}

Recall that $Z$ is the set of the centers of the connected components of $\LAMBDA_{>0}^{\bal}$. 
Define the subset $Z(p^k)$ of $Z$ by 
\begin{align}
Z(p^k) \coloneqq \Set{\zeta \in Z | \DELTA(\zeta) = p^{k}}. 
\end{align}
By the above theorem, $Z = \bigcup_{k \in \mathbb{Z}_{\geq 0}}Z(p^k)$.
Moreover, it follows from \Cref{self-similarity} \ref{ss2} and \Cref{lower degree is divisible by p} \ref{ldb2} that $Z(p^{k}) = p^{k}Z(1)$ for any $k \in \mathbb{Z}_{\geq 0}$.
Therefore, it is enough for us to study $Z(1)$.

\begin{theorem}\label{Z(1)}
	\begin{align}
		Z(1) = \LAMBDA^\bal_\odd \setminus \bigcup_{m > 0}B(p^m\LAMBDA^\bal_\odd,\, p^m),
	\end{align}
	where $B(S,r) = \bigcup_{\zeta \in S}B(\zeta,r)$ for $S \subseteq \LAMBDA$ and $r \in \mathbb{Z}_{> 0}$.
\end{theorem}
\begin{proof}
	First, let $\mu \in Z(1)$.
	Then $\mu \in \LAMBDA^\bal_\odd$.
	Assume that $\mu \in B(p^m\nu,\, p^m)$ for some $\nu \in \LAMBDA^\bal_\odd$ and $m \in \mathbb{Z}_{>0}$.
    Let $C$ be the connected component containing $p^m \nu$. 
    Since $\DELTA(p^m\nu) = p^m\DELTA(\nu) \geq p^m$ by \Cref{self-similarity} \ref{ss1}, 
    \begin{align}
    \mu \in B(p^{m}\nu,p^{m}) \subseteq B(p^{m}\nu,\, \DELTA(p^{m}\nu)) \subseteq C
    \end{align}
    by \Cref{c3}. 
	Since $\mu \in C$ and $\mu \in Z(1)$, we have $C = B(\mu,1)$.
    However, this is a  contradiction.
	Therefore $\mu$ belongs to the right hand side.
	
	Conversely, suppose that $\mu$ belongs to the right-hand side.
    Then $\DELTA(\mu) \geq 1$ since $\mu \in \LAMBDA^\bal_\odd$.
	Therefore $\mu$ belongs to some connected component of $\LAMBDA^\bal_{>0}$.
    Namely, by \Cref{thmcc}, there exist $m \in \mathbb{Z}_{\geq 0}$ and $\zeta \in p^{m}\LAMBDA_{\odd}^{\bal}$ such that $\mu \in B(\zeta,p^m)$.
    This implies $\mu \in B(p^{m}\LAMBDA^\bal_\odd,\, p^m)$. 
	From the assumption, we have $m = 0$ and hence $\mu = \zeta \in Z(1)$.
\end{proof}

\begin{corollary}\label{c210}
	For any $k \in \mathbb{Z}_{\geq 0}$,
	\begin{align}
		Z(p^k) = p^k\LAMBDA^\bal_\odd \setminus \bigcup_{m > k}B(p^m\LAMBDA^\bal_\odd,\, p^m).
	\end{align}
\end{corollary}

\begin{corollary}\label{c211}
	\begin{align}
		\DELTA^{-1}(0) = \LAMBDA^\bal \setminus \bigcup_{m \geq 0}B(p^m\LAMBDA^\bal_\odd,\, p^m).
	\end{align}
\end{corollary}

\subsection{Computation of the exponents}
Given $\mu \in \LAMBDA$ and $k \in \mathbb{Z}_{>0}$, suppose that $\mu = p^k\alpha + \beta$, where $\beta$ satisfies $0 \leq \beta_i < p^k$ for all $i \in \{1,2,3\}$.
To compute $\exp(\Amu)$, we focus on $\beta$.
From the assumption, $\mu$ belongs to the set $C(\alpha,p^k)$ given by
\begin{align}
	C(\alpha,p^k) \coloneqq \Set{\mu \in \LAMBDA | p^k\alpha_i \leq \mu_i < p^k(\alpha_i+1) \text{ for all $i \in\{1,2,3\}$}}.
\end{align}
The set $C(\alpha,p^k)$ is contained in the cube of side length $p^k$ with the vertex set
\begin{align}
	\Set{ \alpha + \sum_{i \in I}p^k\bm{e}_i | I \subseteq \{1,2,3\}},
\end{align}
where $\bm{e}_1 = (1,0,0)$, $\bm{e}_2 = (0,1,0)$, and $\bm{e}_3 = (0,0,1)$.
Note that four points of the eight vertices above belong to $p^k\LAMBDA_\even$ and the others belong to $p^k\LAMBDA_\odd$. 

Define the set $T(p^k)$ by 
\begin{align}
    T(p^k) &\coloneqq \Set{x_1(\bm{e}_2+\bm{e}_3) + x_2(\bm{e}_1+\bm{e}_3) + x_3(\bm{e}_1+\bm{e}_2) \in \LAMBDA | \begin{lgathered}
        x_1,x_2,x_3 \geq 0,\\ x_1 + x_2 + x_3 \leq p^k
    \end{lgathered}}. 
\end{align}
Then $T(p^k)$ is the set of lattice points of the regular tetrahedron with the vertex set (see \Cref{fig:tetrahedron1})
\begin{align}
	\Bigl\{ \bm{0},\ p^k(\bm{e}_2 + \bm{e}_3),\ p^k(\bm{e}_1 + \bm{e}_3),\ p^k(\bm{e}_1 + \bm{e}_2)\Bigr\}.
\end{align}
Moreover, $p^k\bm{1} - T(p^k)$ is the set of lattice points of the regular tetrahedron with the vertex set (see \Cref{fig:tetrahedron2})
\begin{align}
	\Bigl\{ p^k\bm{e}_1,\ p^k\bm{e}_2,\ p^k\bm{e}_3,\ p^k\bm{1}\Bigr\}.
\end{align}
\begin{figure}[t]
    \centering
    \begin{tikzpicture}[scale = 1.5]
        \node[coordinate] (000) at (0,0){};
        \node[coordinate] (100) at (2,0){};
        \node[coordinate] (010) at (1+1/8,1-1/2){};
        \node[coordinate] (110) at (3+1/8,1-1/2){};
        \node[coordinate] (001) at (0,2){};
        \node[coordinate] (101) at (2,2){};
        \node[coordinate] (011) at (1+1/8,3-1/2){};
        \node[coordinate] (111) at (3+1/8,3-1/2){};

        \draw (000)  node [below left] {$p^k\alpha$};
        \draw (100)  node [below] {$p^k\alpha + p^k\bm{e}_1$};
        \draw (010)  node [above right] {$p^k\alpha + p^k\bm{e}_2$};
        \draw (001)  node [above left] {$p^k\alpha + p^k\bm{e}_3$};
        \draw (111)  node [above right] {$p^k\alpha + p^k\bm{1}$};
        
        \draw (101) -- (100) -- (000) -- (001) -- (101) -- (111) -- (011) -- (001);
        \draw (100) -- (110) -- (111);
        \draw[dashed] (000) -- (010) -- (110);
        \draw[dashed] (010) -- (011);
        \draw[thick] (000) -- (011) -- (101) -- (000) -- (110) -- (101);
        \draw[thick, dashed] (110) -- (011);
    \end{tikzpicture}
    \caption{The tetrahedron $p^k\alpha + T(p^k)$} \label{fig:tetrahedron1}

    \ \vspace{3mm}
    
    \begin{tikzpicture}[scale = 1.5]
        \node[coordinate] (000) at (0,0){};
        \node[coordinate] (100) at (2,0){};
        \node[coordinate] (010) at (1+1/8,1-1/2){};
        \node[coordinate] (110) at (3+1/8,1-1/2){};
        \node[coordinate] (001) at (0,2){};
        \node[coordinate] (101) at (2,2){};
        \node[coordinate] (011) at (1+1/8,3-1/2){};
        \node[coordinate] (111) at (3+1/8,3-1/2){};

        \draw (000)  node [below left] {$p^k\alpha$};
        \draw (100)  node [below] {$p^k\alpha + p^k\bm{e}_1$};
        \draw (010)  node [left] {$p^k\alpha + p^k\bm{e}_2$};
        \draw (001)  node [above left] {$p^k\alpha + p^k\bm{e}_3$};
        \draw (111)  node [above right] {$p^k\alpha + p^k\bm{1}$};
        
        \draw (101) -- (100) -- (000) -- (001) -- (101) -- (111) -- (011) -- (001);
        \draw (100) -- (110) -- (111);
        \draw[dashed] (000) -- (010) -- (110);
        \draw[dashed] (010) -- (011);
        \draw[thick] (001) -- (100) -- (111) -- (001) -- (010) -- (100);
        \draw[thick, dashed] (111) -- (010);
    \end{tikzpicture}
    \caption{The tetrahedron $p^k\alpha+p^k\boldsymbol{1}-T(p^k)$}
    \label{fig:tetrahedron2}
\end{figure}

\begin{lemma}\label{l21}
	Given $\mu \in \LAMBDA$ and $k \in \mathbb{Z}_{>0}$, suppose that $\mu = p^k\alpha + \beta \ \ (0 \leq \beta_i < p^k)$.
	\begin{enumerate}[label=(\arabic*)]
		\item \label{l211} If $\mu \not\in B(p^k\LAMBDA_\odd,\, p^k)$, then one of the following holds:
		\begin{enumerate}[label=(\roman*)]
			\item \label{l21i} $\alpha \in \LAMBDA_\even$ and $\max\beta \leq \dfrac{|\beta|}{2} \leq p^k$;
			\item \label{l21ii} $\alpha \in \LAMBDA_\odd$ and $\max(p^k\bm{1} - \beta) \leq \dfrac{|p^k\bm{1}-\beta|}{2} \leq p^k$, where $\bm{1} = (1,1,1)$.
		\end{enumerate}
		\item \label{l212} If $\mu \in B(p^k\LAMBDA_\odd,\, p^k)$, then one of the following holds:
		\begin{enumerate}[label=(\roman*),start=3]
			\item \label{l21iii} $\alpha \in \LAMBDA_\even$, $\beta_i > \dfrac{|\beta|}{2}$ and $\mu \in B(p^k(\alpha + \bm{e}_i),\, p^k)$ for some $i \in \{1,2,3\}$;
			\item \label{l21iv} $\alpha \in \LAMBDA_\even$, $\dfrac{|\beta|}{2} > p^k$ and $\mu \in B(p^k(\alpha + \bm{1}),\, p^k)$;
			\item \label{l21v} $\alpha \in \LAMBDA_\odd$, $p^k - \beta_i > \dfrac{|p^k\bm{1} - \beta|}{2}$ and $\mu \in B(p^k(\alpha + \bm{e}_{i+1} + \bm{e}_{i+2}),\, p^k)$ for some $i \in \{1,2,3\}$;
			\item \label{l21vi} $\alpha \in \LAMBDA_\odd$, $\dfrac{|p^k\bm{1} - \beta|}{2} > p^k$ and $\mu \in B(p^k\alpha ,\, p^k)$.
		\end{enumerate}	
	\end{enumerate}
\end{lemma}
\begin{proof}
     Consider the following representation for $\beta$:
    \begin{align}
        \beta = x_1(\bm{e}_2+\bm{e}_3) + x_2(\bm{e}_1+\bm{e}_3) + x_3(\bm{e}_1+\bm{e}_2) = (x_2 + x_3,\, x_1 + x_3,\, x_1 + x_2).
    \end{align}
    
	\ref{l211} Take $\mu \in \LAMBDA \setminus B(p^k\LAMBDA_\odd,\, p^k)$.
	First, suppose that $\alpha \in \LAMBDA_\even$.
	Then we have 
	\begin{align}
		C(\alpha,p^k)\setminus B(p^k\LAMBDA_\odd,\, p^k) = p^k\alpha + T(p^k) = \Set{p^k\alpha + \tau | \tau \in T(p^k)} \label{s1}
	\end{align}
	and hence $\beta \in T(p^k)$.
	Since it satisfies $x_1,x_2,x_3 \geq 0$ and $x_1 + x_2 + x_3 \leq p^k$, then $\mu$ satisfies \ref{l21i}.
	
	Second, suppose that $\alpha \in \LAMBDA_\odd$.
	Then we have
	\begin{align}
		C(\alpha,p^k)\setminus B(p^k\LAMBDA_\odd,\, p^k) = p^k\alpha + p^k\bm{1} - T(p^k) = \Set{p^k\alpha + p^k\bm{1} - \tau | \tau \in T(p^k)} \label{s2}
	\end{align}
	and hence $p^k\bm{1} - \beta \in T(p^k)$.
	Thus $\mu$ satisfies \ref{l21ii}.
	
	\ref{l212} Take $\mu \in B(p^k\LAMBDA_\odd,\, p^k)$.
	First, suppose that $\alpha \in \LAMBDA_\even$.
	By \eqref{s1}, we have $\beta \not\in T(p^k)$ and hence it does not satisfy either $x_1,x_2,x_3 \geq 0$ or $x_1 + x_2 + x_3 \leq p^k$.
	If $x_i < 0$ for some $i \in \{1,2,3\}$, then it satisfies that 
	\begin{align}
		\max{\beta} \geq \beta_i = x_{i+1} + x_{i+2} > x_i + x_{i+1} + x_{i+2} = \dfrac{|\beta|}{2}
	\end{align}
	and furthermore that $\beta \in B(p^k\bm{e}_i,\, p^k)$.
	Thus $\mu$ satisfies \ref{l21iii}.
	If $x_1 + x_2 + x_3 > p^k$, then it satisfies that 
	\begin{align}
		\dfrac{|\beta|}{2} = x_1 + x_2 + x_3 > p^k
	\end{align}
	and furthermore that $\beta \in B(p^k\bm{1},\, p^k)$.
	Therefore $\mu$ satisfies \ref{l21iv}.
	
	Second, suppose that $\alpha \in \LAMBDA_\odd$.
	By \eqref{s2}, we have $p^k\bm{1} - \beta \not\in T(p^k)$ and hence either of the following is satisfied:
	\begin{itemize}
		\item $p^k-\beta_i > \dfrac{|p^k\bm{1}-\beta|}{2}$ \quad and \quad  $p^k\bm{1}-\beta \in B(p^k\bm{e}_i,\, p^k)$;
		\item $\dfrac{|p^k\bm{1}-\beta|}{2} > p^k$ \quad and \quad $p^k\bm{1} - \beta \in B(p^k\bm{1},\, p^k)$.
	\end{itemize}
	Therefore $\mu$ satisfies \ref{l21v} or \ref{l21vi}.
\end{proof}

The following theorem allows us to compute $\exp(\Amu)$ algorithmically. 
\begin{theorem}\label{main thm:expnonents}
	Let $\mu \in \LAMBDA^\bal$. 
    Define 
    \begin{align}
        k \coloneqq \max\left(\{0\} \cup \Set{m \in \mathbb{Z}_{> 0} | p^{m} \leq \dfrac{|\mu|}{2}, \ \mu \in B(p^m\LAMBDA^\bal_\odd,\, p^m)}\right).
    \end{align}
	\begin{enumerate}[label=(\arabic*)]
		\item \label{a1} Suppose that $k=0$.
		\begin{enumerate}[label=(\alph*)]
			\item \label{aa} If $\mu \in \LAMBDA^\bal_\odd$, then 
			\begin{align}
				\DELTA(\mu) = 1 \text{\quad and \quad} \exp(\Amu) = \left( \dfrac{|\mu|-1}{2},\ \dfrac{|\mu| + 1}{2}\right);
			\end{align}
			
			\item \label{ab} If $\mu \in \LAMBDA^\bal_\even$, then 
			\begin{align}
				\DELTA(\mu) = 0 \text{\quad and \quad} \exp(\Amu) = \left( \dfrac{|\mu|}{2},\ \dfrac{|\mu|}{2} \right).
			\end{align}
			
		\end{enumerate}
		\item \label{a2} Suppose that $k>0$ and $\mu = p^k\alpha + \beta \ (0 \leq \beta_i < p^k)$.
		\begin{enumerate}[label=(\alph*),start=3]
			\item \label{ac} If $\alpha \in \LAMBDA_\even$ and $\beta_i  > \dfrac{|\beta|}{2}$ (\ref{l21iii} is satisfied), then 
			\begin{align}
				&\DELTA(\mu) = \beta_i - \beta_{i+1} - \beta_{i+2} \\&\text{and \quad} \exp(\Amu) = \left( \dfrac{|\alpha|}{2}p^k + \beta_{i+1} + \beta_{i+2},\ \dfrac{|\alpha|}{2}p^k + \beta_i \right);
			\end{align}
			\item \label{ad} If $\alpha \in \LAMBDA_\even$ and $\dfrac{|\beta|}{2} > p^k$ (\ref{l21iv} is satisfied), then 
			\begin{align}
				\DELTA(\mu) = |\beta| - 2p^k \text{\quad and \quad} \exp(\Amu) = \left( \dfrac{|\alpha|}{2}p^k + p^k,\ \dfrac{|\alpha|}{2}p^k + |\beta|-p^k \right);
			\end{align}
			\item \label{ae} If $\alpha \in \LAMBDA_\odd$ and $p^k - \beta_i > \dfrac{|p^k\bm{1} - \beta|}{2}$ (\ref{l21v} is satisfied), then
			\begin{align}
				&\DELTA(\mu) =  - \beta_i + \beta_{i+1} + \beta_{i+2} - p^k \\
                &\text{and \quad} \exp(\Amu) = \left( \dfrac{|\alpha|+1}{2}p^k + \beta_i,\ \dfrac{|\alpha|-1}{2}p^k + \beta_{i+1} + \beta_{i+2} \right);
			\end{align}
			\item \label{af} If $\alpha \in \LAMBDA_\odd$ and $\dfrac{|p^k\bm{1} - \beta|}{2} > p^k$ (\ref{l21vi} is satisfied), then
			\begin{align}
				\DELTA(\mu) =  p^k - |\beta| \text{\quad and \quad} \exp(\Amu) = \left( \dfrac{|\alpha|-1}{2}p^k + |\beta|,\ \dfrac{|\alpha|+1}{2}p^k \right).
			\end{align}
		\end{enumerate}
	\end{enumerate}
\end{theorem}
\begin{proof}
	\ref{aa} If $p^m > \frac{|\mu|}{2}$, then $\mu \not\in B(p^m\LAMBDA_{\odd}^{\bal},p^m)$. 
    Since $k=0$, $\mu \not\in B(p^m\LAMBDA_{\odd}^{\bal},p^m)$ for any $m > 0$. 
    If $\mu \in \LAMBDA_{\odd}^{\bal}$, then $\DELTA(\mu)=1$ by \Cref{Z(1)}. 
	
	\ref{ab} This follows from \Cref{c211} by the same argument as above.
	
	\ref{ac} By \Cref{l21}, $\mu$ belongs to $B(p^k(\alpha + \bm{e}_i),\, p^k)$.
	Since $p^k(\alpha + \bm{e}_i) \in p^k\LAMBDA_\odd^\bal$ and $\mu \not\in \bigcup_{m>k}B(p^m\LAMBDA^\bal_\odd,\, p^m)$, we have $p^k(\alpha + \bm{e}_i) \in Z(p^k)$ by \Cref{c210}.
	Hence
	\begin{align}
		\DELTA(\mu) &= \DELTA(p^k(\alpha + \bm{e}_i)) - |\mu - p^k(\alpha + \bm{e}_i)|\\
		&= p^k - |\beta - p^k\bm{e}_i|\\
		&= p^k - (p^k - \beta_i + \beta_{i+1} + \beta_{i+2})\\
		&= \beta_i - \beta_{i+1} - \beta_{i+2}.
	\end{align}
	
	\ref{ad} By \Cref{l21}, $\mu$ belongs to $B(p^k(\alpha + \bm{1}),\, p^k)$.
	As above, we have $p^k(\alpha + \bm{1}) \in Z(p^k)$.
	Hence
	\begin{align}
		\DELTA(\mu) &= \DELTA(p^k(\alpha + \bm{1})) - |\mu - p^k(\alpha + \bm{1})|\\
		&= p^k - |\beta - p^k\bm{1}|\\
		&= p^k - (3p^k - |\beta|)\\
		&= |\beta| - 2p^k.
	\end{align}
	
	\ref{ae} By \Cref{l21}, $\mu$ belongs to $B(p^k(\alpha + \bm{e}_{i+1} + \bm{e}_{i+2}),\, p^k)$.
	As above, we have $p^k(\alpha + \bm{e}_{i+1} + \bm{e}_{i+2}) \in Z(p^k)$.
	Hence
	\begin{align}
		\DELTA(\mu) &= \DELTA(p^k(\alpha + \bm{e}_{i+1} + \bm{e}_{i+2})) - |\mu - p^k(\alpha + \bm{e}_{i+1} + \bm{e}_{i+2})|\\
		&= p^k - |\beta - p^k(\bm{e}_{i+1} +\bm{e}_{i+2})|\\
		&= p^k - (2p^k + \beta_i - \beta_{i+1} - \beta_{i+2})\\
		&= -\beta_i + \beta_{i+1} + \beta_{i+2} - p^k.
	\end{align}
	
	\ref{af} By \Cref{l21}, $\mu$ belongs to $B(p^k\alpha,\, p^k)$.
	As above, we have $p^k\alpha \in Z(p^k)$.
	Hence
	\begin{align}
		\DELTA(\mu) &= \DELTA(p^k\alpha) - |\mu - p^k\alpha|\\
		&= p^k - |\beta|.
	\end{align}
\end{proof}

\begin{example} \label{eg31}
Consider the case $p=3$ and $\mu=(41,52,31) \in \LAMBDA^\bal$. 
Then $k=3$ and $p^{k}=3^{3} = 27$. 
\begin{align}
(41,52,31) = 27(1,1,1) + (14,25,4). 
\end{align}
Thus $\alpha = (1,1,1) \in \LAMBDA_\odd$ and $\beta = (14,25,4)$. 
Moreover, 
\begin{align}
p^{3}\boldsymbol{1}-\beta &= (27,27,27) - (14,25,4) = (13,2,23), \\
\dfrac{|p^{3}\boldsymbol{1}-\beta|}{2} &= \dfrac{38}{2}=19 < 23 = p^{3}-\beta_{3}. 
\end{align}
Therefore
\begin{align}
\DELTA(\mu) &= -\beta_{3}+\beta_{1}+\beta_{2}-p^{3} = -4+14+25-27 = 8,
\end{align}
and 
\begin{align}
\exp(\A,\mu) &= \left( \dfrac{|\alpha|+1}{2}p^3 + \beta_3,\ \dfrac{|\alpha|-1}{2}p^3 + \beta_{1} + \beta_{2}\right) \\ 
&= (2\cdot 27 + 4,\, 27 + 14 + 25)\\ 
&= (58,66). 
\end{align}
In this case, $\mu$ belongs to the connected component with radius $27$ and center $(54,54,27)$.

\end{example}

\section{Periodicity and symmetry of exponents} \label{section:plane}

\subsection{Periodicity}

For $d > 0$, define a map $\PSI_d:\Der_S \longrightarrow \Der_S$ by
\begin{align}
    \PSI_d : \theta = \theta(x)\partial_x + \theta(y)\partial_y \ \longmapsto \  \theta(x)x^{p^d}\partial_x - \theta(y)y^{p^d}\partial_y.
\end{align}

\begin{lemma}
    Let 
    $\mu \in \LAMBDA$.
    If an integer $d > 0$ satisfies $\mu_3 \leq p^d$, then 
    \begin{align}
        \PSI_d\bigl(D(\Amu)\bigr) \subseteq D(\A,\, \mu + (p^d,p^d,0)).
    \end{align}
\end{lemma}
\begin{proof}
    Suppose that $\theta \in D(\Amu)$.
    Then 
    $\PSI_d(\theta)(x) = \theta(x)x^{p^d} \in x^{\mu_1+p^d}S$, $\PSI_d(\theta)(y) = -\theta(y)y^{p^d} \in y^{\mu_2+p^d}S$.
    Therefore
    \begin{align}
        \PSI_d(\theta)(x+y) &= \theta(x)x^{p^d} - \theta(y)y^{p^d}\\
        &= \theta(x)x^{p^d} + \theta(x)y^{p^d} - \theta(x)y^{p^d} - \theta(y)y^{p^d}\\
        &= \theta(x)\Bigl(x^{p^d} + y^{p^d}\Bigr) - \Bigl(\theta(x) + \theta(y)\Bigr)y^{p^d}\\
        &= \theta(x)(x+y)^{p^d} - \Bigl(\theta(x) + \theta(y)\Bigr)y^{p^d},
    \end{align}
    and hence $\PSI_d(\theta)(x+y) \in (x+y)^{\mu_3}S$.
    Thus $\PSI_d(\theta) \in D(\A,\, \mu + (p^d, p^d, 0))$.
\end{proof}

The map $\PSI_d$ preserves a basis as follows.

\begin{theorem}[Periodicity theorem] \label{Period2}
    Let $\mu \in \LAMBDA$.
    Suppose that $d >0$ satisfies $\mu_3 \leq p^d$.
    Then the following are equivalent:
    \begin{enumerate}[label=(\roman*)]
        \item $\{\theta,\theta'\}$ is a basis for $D(\Amu)$;
        \item $\{\PSI_d(\theta),\, \PSI_d(\theta')\}$ is a basis for $D(\A,\, \mu + (p^d, p^d, 0))$.
    \end{enumerate}
    Furthermore, we have $\DELTA(\mu) = \DELTA(\mu + (p^d, p^d, 0))$.
\end{theorem}
\begin{proof}
    Suppose that $\{\theta,\theta'\}$ is a basis for $D(\Amu)$.
    Then $M(\theta,\theta') \doteq Q(\A,\mu)$ (Theorem \ref{Saito-Ziegler}).
    We have
    \begin{align}
        M(\PSI_d(\theta),\, \PSI_d(\theta'))
         &= -\theta(x)x^{p^d}\theta'(y)y^{p^d} + \theta(y)y^{p^d}\theta'(x)x^{p^d}\\
         &= -x^{p^d}y^{p^d}M(\theta,\theta')\\
         &\doteq x^{p^d}y^{p^d}Q(\Amu)\\
         &\doteq Q(\A,\, \mu + (p^d, p^d, 0)),
    \end{align}
    and hence $\{\PSI_d(\theta),\, \PSI_d(\theta')\}$ is a basis for $D(\A,\, \mu + (p^d, p^d, 0))$.

    Suppose that $\{\PSI_d(\theta),\, \PSI_d(\theta')\}$ is a basis for $D(\A,\, \mu + (p^d, p^d, 0))$.
    Then 
    \begin{align}
	   M(\PSI_d(\theta),\, \PSI_d(\theta')) \doteq Q(\A,\, \mu + (p^d, p^d, 0)) = x^{p^d}y^{p^d}Q(\Amu).
    \end{align}
    Since
    \begin{align}
        \theta = \dfrac{\PSI_d(\theta)(x)}{x^{p^d}}\partial_x - \dfrac{\PSI_d(\theta)(y)}{y^{p^d}}\partial_y,\quad 
        \theta' = \dfrac{\PSI_d(\theta')(x)}{x^{p^d}}\partial_x - \dfrac{\PSI_d(\theta')(y)}{y^{p^d}}\partial_y,
    \end{align}
    we have
    \begin{align}
        M(\theta,\theta')
         &= -\dfrac{\PSI_d(\theta)(x)}{x^{p^d}}\dfrac{\PSI_d(\theta')(y)}{y^{p^d}} + \dfrac{\PSI_d(\theta)(y)}{y^{p^d}}\dfrac{\PSI_d(\theta')(x)}{x^{p^d}}\\
         &= -\dfrac{M(\PSI_d(\theta),\, \PSI_d(\theta'))}{x^{p^d}y^{p^d}}\\
         &\doteq Q(\A,\mu).
    \end{align}
    Hence $\{\theta,\theta'\}$ is a basis for $D(\Amu)$.
\end{proof}

\begin{example}\label{eg:334p5}
Let $p = 5$ and $\mu = (3,3,4)$.
Then
\begin{align}
    \theta_\mu &= (x^5 - x^4y + x^3y^2)\partial_x - (x^2y^3 - xy^4)\partial_y,\\
    \theta_\mu' &= x^5\partial_x + y^5\partial_y
\end{align}
form a basis for $D(\A,\mu)$  (See \Cref{eg334}).
Define $\nu \coloneqq \mu + (5,5,0) = (8,8,4)$.
\Cref{Period2} implies that $\DELTA(\nu) = \DELTA(\mu) = 0$ and 
\begin{align}
    \theta_\nu &= \PSI_1(\theta_\mu) = (x^{10} - x^9y + x^8y^2)\partial_x + (x^2y^8 - xy^9)\partial_y,\\
    \theta_\nu' &= \PSI_1(\theta_\mu') = x^{10}\partial_x - y^{10}\partial_y
\end{align}
form a basis for $D(\A,\nu)$.
In fact, we have
\begin{align}
    \det{M(\theta_\nu,\theta_\nu')} = x^{12}y^8 - x^{11}y^9 + x^{10}y^{10} - x^9y^{11} + x^8y^{12} = x^8y^8(x+y)^4.
\end{align}
\end{example}

\subsection{Symmetry}

Let $d \in \mathbb{Z}_{>0}$, define 
\begin{align}
    \LAMBDA_{\leq p^d} \coloneqq \Set{ (\mu_1,\mu_2,\mu_3) \in \LAMBDA | \mu_1,\mu_2,\mu_3 \leq p^d }
\end{align}
For $\mu = (\mu_1,\mu_2,\mu_3) \in \LAMBDA_{\leq p^d}$, define $\mu^\vee \in \LAMBDA_{\leq p^d}$ by 
\begin{align}
    \mu^\vee = (\mu_1^\vee,\mu_2^\vee,\mu_3^\vee) \coloneqq (p^d-\mu_1,\, p^d-\mu_2,\, \mu_3).
\end{align}
It is clear that $(\mu^\vee)^\vee = \mu$ for any $\mu \in \LAMBDA_{\leq p^d}$.

Let $\mu \in\LAMBDA_{\leq p^d}$.
We can write $\theta \in D(\A,\mu)$ as
\begin{align}
    \theta = f x^{\mu_1}\partial_x + g y^{\mu_2}\partial_y \in D(\A,\mu),
\end{align}
where $f $, $g  \in S$.
Define 
\begin{align}
    \theta^\vee \coloneqq \dfrac{g y^{\mu_2}x^{p^d}\partial_x - f x^{\mu_1}y^{p^d}\partial_y}{x^{\mu_1}y^{\mu_2}} = g x^{\mu_1^\vee}\partial_x - f y^{\mu_2^\vee}\partial_y.
\end{align}
Then $(\theta^\vee)^\vee = -\theta$ holds.

\begin{lemma}
Let $\theta = f x^{\mu_1}\partial_x + g y^{\mu_2}\partial_y \in D(\A,\mu)$.
Then $\theta^\vee \in D(\A,\mu^\vee)$.
\end{lemma}
\begin{proof}
    It is clear that $\theta^\vee(x) \in x^{\mu_1^\vee}S$ and $\theta^\vee(y) \in y^{\mu_2^\vee}S$.
    Since $\theta \in D(\A,\mu)$,
    there exists $h \in S$ such that
    \begin{align}
        f x^{\mu_1} + g y^{\mu_2} = h (x+y)^{\mu_3}.\label{h}
    \end{align}
    Then
    \begin{align}
        \theta^\vee(x+y) &= \dfrac{g y^{\mu_2}x^{p^d} - f x^{\mu_1}y^{p^d}}{x^{\mu_1}y^{\mu_2}} \\
        &= \dfrac{g y^{\mu_2}x^{p^d} + f x^{\mu_1}x^{p^d} - f x^{\mu_1}x^{p^d}- f x^{\mu_1}y^{p^d}}{x^{\mu_1}y^{\mu_2}}\\
        &= \dfrac{g y^{\mu_2}x^{p^d} + f x^{\mu_1}x^{p^d} - f x^{\mu_1}x^{p^d}- f x^{\mu_1}y^{p^d}}{x^{\mu_1}y^{\mu_2}}\\
        &= \dfrac{\bigl(g y^{\mu_2} + f x^{\mu_1}\bigr)x^{p^d} - f x^{\mu_1}(x^{p^d} + y^{p^d})}{x^{\mu_1}y^{\mu_2}}\\
        &= \dfrac{h (x+y)^{\mu_3}x^{p^d} - f x^{\mu_1}(x + y)^{p^d}}{x^{\mu_1}y^{\mu_2}}\\
        &= (x + y)^{\mu_3} \cdot \dfrac{h x^{p^d} - f x^{\mu_1}(x + y)^{p^d-\mu_3}}{x^{\mu_1}y^{\mu_2}}\\
        &= (x + y)^{\mu_3} \cdot \dfrac{h x^{p^d-\mu_1} - f (x + y)^{p^d-\mu_3}}{y^{\mu_2}},
    \end{align}
    and 
    \begin{align}
        \theta^\vee(x+y) &= \dfrac{g y^{\mu_2}x^{p^d} - f x^{\mu_1}y^{p^d}}{x^{\mu_1}y^{\mu_2}} \\
        &= \dfrac{g y^{\mu_2}x^{p^d} + g y^{\mu_2}y^{p^d} - g y^{\mu_2}y^{p^d} - f x^{\mu_1}y^{p^d}}{x^{\mu_1}y^{\mu_2}}\\
        &= \dfrac{g y^{\mu_2}(x^{p^d} + y^{p^d}) - \bigl(g y^{\mu_2} + f x^{\mu_1} \bigr) y^{p^d}}{x^{\mu_1}y^{\mu_2}}\\
        &= \dfrac{g y^{\mu_2}(x + y)^{p^d} - h (x + y)^{\mu_3} y^{p^d}}{x^{\mu_1}y^{\mu_2}}\\
        &= (x + y)^{\mu_3} \cdot \dfrac{g y^{\mu_2}(x + y)^{p^d - \mu_3} - h  y^{p^d}}{x^{\mu_1}y^{\mu_2}}\\
        &= (x + y)^{\mu_3} \cdot \dfrac{g (x + y)^{p^d - \mu_3} - h  y^{p^d-\mu_2}}{x^{\mu_1}}.
    \end{align}
    Thus we have
    \begin{align}
        x^{\mu_1}\bigl( h x^{p^d - \mu_1} - f (x + y)^{p^d - \mu_3} \bigr) = y^{\mu_2} \bigl( g (x + y)^{p^d - \mu_3} - h y^{p^d - \mu_2} \bigr)
    \end{align}
    This implies that $h x^{p^d - \mu_1} - f (x + y)^{p^d - \mu_3}$ can be divided by $y^{\mu_2}$.
    Therefore $\theta^\vee(x+y) \in (x+y)^{\mu_3}S$.
\end{proof}

The operator $\vee$ preserves a basis as follows.
\begin{theorem}\label{veebasis}
    Let $\mu \in \LAMBDA_{\leq p^d}$.
    Then the following are equivalent:
    \begin{enumerate}[label = (\roman*)]
        \item $\{\theta ,\theta'\}$ is a basis for $D(\A,\mu)$;
        \item $\{\theta^\vee, {\theta'}^\vee\}$ is a basis for $D(\A,\mu^\vee)$.
    \end{enumerate}
    Furthermore, we have $\DELTA(\mu) = \DELTA(\mu^\vee)$.
\end{theorem}
\begin{proof}
    Write $\theta$ and $\theta'$ as 
    \begin{align}
        \theta = f  x^{\mu_1} \partial_x + g  y^{\mu_2} \partial_y,\qquad 
        \theta' = f'  x^{\mu_1} \partial_x + g'  y^{\mu_2} \partial_y.
    \end{align}
    If $\{\theta,\theta'\}$ is a basis for $D(\A,\mu)$, then 
    it follows from Saito's criterion (\Cref{Saito-Ziegler}) that
    \begin{align}
        &x^{\mu_1}y^{\mu_2}\bigl( f  \cdot g'  - f'  \cdot g  \bigr) \\
        & \quad = 
        f x^{\mu_1} \cdot g' y^{\mu_2} - f' x^{\mu_1} \cdot g y^{\mu_2}\\
        & \quad = \det{M(\theta,\theta')} \\
        & \quad \doteq x^{\mu_1}y^{\mu_2}(x+y)^{\mu_3}.
    \end{align}
    We have 
    \begin{align}
        \det{M(\theta^\vee,\, {\theta'}^\vee)}
        &= - g x^{\mu_1^\vee} \cdot f' y^{\mu_2^\vee} + g' x^{\mu_1^\vee} \cdot f y^{\mu_2^\vee}\\
        &= - x^{\mu_1^\vee}y^{\mu_2^\vee} \bigl( g  \cdot f'  - g'  \cdot f  \bigr)\\
        &\doteq x^{\mu_1^\vee}y^{\mu_2^\vee}(x + y)^{\mu_3}.
    \end{align}
    Hence, by Saito's criterion, $\{\theta^\vee, {\theta'}^\vee\}$ can be a basis for $D(\A,\mu^\vee)$.
    The proof is complete.
\end{proof}

\section{The basis with binomial coefficients} \label{section:gammam}

For $m \in \mathbb{Z}_{> 0}$, we express the base-$p$ expansion of $m$ as
\begin{align}
    m = \sum_{e = 0}^{\infty}c_e(m)p^e,
\end{align}
where $0 \leq c_e(m) < p$.
Define $s(m) \coloneqq \min\Set{ e \in \mathbb{Z}_{\geq 0} | c_e(m) \neq 0 }$.
For example, if $p = 3$ and $m = 16 = 1 \cdot 3^0 + 2 \cdot 3 + 1 \cdot 3^2$, then 
\begin{align}
    c_e(m) = \begin{cases*}
        2 & if $e = 1$;\\
        1 & if $e = 0$ or $2$;\\
        0 & otherwise,
    \end{cases*}
    \qquad s(m) = 0.
\end{align}

The following formula is known for the binomial coefficients.
\begin{lemma}[Lucas]
    For $m,j \in \mathbb{Z}_{\geq 0}$, the binomial coefficient $\binom{m}{j}$ has the following relation:
    \begin{align}
        \binom{m}{j} \equiv \prod_{e = 0}^\infty \binom{c_e(m)}{c_e(j)} \pmod{p},
    \end{align}
    where we use the convention that $\binom{a}{b} = 0$ if $b > a$.
    Hence $\binom{m}{j} \equiv 0 \pmod{p}$ if and only if $c_e(j) > c_e(m)$ for some $e \in \mathbb{Z}_{\geq 0}$.
\end{lemma}

\begin{table}[h]
\centering
{\renewcommand\arraystretch{1.5}
    \caption{Values of binomial coefficients for $m = 16$ and $p = 3$}
    
    \begin{tabular}{c|ccccccccccccccccc}
    $j$                     & 0 & 1 & 2 & 3 & 4 & 5 & 6 & 7 & 8 & 9 & 10 & 11 & 12 & 13 & 14 & 15 & 16 \\ \hline
    $\binom{m}{j} \pmod{p}$ & 1 & 1 & 0 & 2 & 2 & 0 & 1 & 1 & 0 & 1 & 1  & 0  & 2  & 2  & 0  & 1  & 1 
    \end{tabular}
    
    \label{table:binom}
    }
\end{table}

Let $m \in \mathbb{Z}_{>0}$ and
\begin{align}
    \LAMBDA(m) \coloneqq \Set{(\mu_1,\mu_2,\mu_3) \in \LAMBDA | \mu_3 = m}.
\end{align}
For $\mu \in \LAMBDA(m)$, define 
\begin{align}
        \psi_\mu \coloneqq \sum_{j = \mu_1}^{m}\binom{m}{j}x^jy^{m-j}\partial_x + \sum_{j=0}^{\mu_1-1}\binom{m}{j}x^jy^{m-j}\partial_y,\qquad 
        \psi'_\mu \coloneqq x^{\mu_1}y^{\mu_2}(\partial_y - \partial_x).
\end{align}
In particular, if $\mu_1 = 0$, then $\psi_\mu = (x+y)^m\partial_x$.
If $\mu_1 > m$, then $\psi_\mu = (x + y)^m\partial_y$.
We consider the set
\begin{align}
    \GAMMA(m) \coloneqq \Set{\mu \in \LAMBDA(m) |  \text{$\{\psi_\mu,\psi_\mu'\}$ is a basis for $D(\A,\mu)$}}.
\end{align}
Since 
\begin{align}
	\psi_\mu' \in D(\Amu),\quad \psi_\mu(x) \in x^{\mu_1}S,\quad \psi_\mu(x+y) = (x+y)^{m} \in (x+y)^{m}S
\end{align}
and 
\begin{align}
	\det{M(\psi_\mu,\psi_\mu')} = x^{\mu_1}y^{\mu_2}\sum_{j = \mu_1}^{m}\binom{m}{j}x^jy^{m-j} + x^{\mu_1}y^{\mu_2}\sum_{j=0}^{\mu_1-1}\binom{m}{j}x^jy^{m-j}
    = x^{\mu_1}y^{\mu_2}(x+y)^{m}
\end{align}
hold, $\mu \in \GAMMA(m)$ if and only if $\psi_\mu(y)$ is divisible by $y^{\mu_2}$.

In particular, it is easy to see the following fact.
\begin{proposition} \label{mu=0}
    If $\mu \in \GAMMA(m)$ satisfies $\mu_1 = 0$ or $\mu_2 = 0$, then $\mu \in \GAMMA(m)$.
\end{proposition}
\begin{proof}
    If $\mu_1 = 0$, then $\mu \in \GAMMA(m)$ since $\psi_\mu(y) = 0$.
    Hence $\psi_\mu(y)$ is divisible by $y^{\mu_2}$.
    If $\mu_2 = 0$, then $\mu \in \GAMMA(m)$ since $y^{\mu_2} = 1$.
    Hence $\psi_\mu(y)$ is divisible by $y^{\mu_2}$.
\end{proof}

For $\mu \in \LAMBDA(m)$, define
\begin{align}
    J_\mu \coloneqq \Set{j \in \mathbb{Z}_{\geq 0} | m - \mu_2 < j < \mu_1}.
\end{align}
We can characterize $\GAMMA(m)$ with binomial coefficients as follows.
\begin{theorem}\label{GAMMA_p(m)}
    For any $m \in \mathbb{Z}_{>0}$, we have
    \begin{align}
        \GAMMA(m) &= \Set{\mu \in \LAMBDA(m) | j \in J_\mu \implies \binom{m}{j} \equiv 0 \pmod{p}}.
    \end{align}
\end{theorem}

\begin{proof}
    Let $\mu \in \GAMMA(m)$.
    If $\mu_1 = 0$, then $J_\mu = \emptyset$. Hence $\mu$ belongs to the right-hand set.
    If $\mu_2 = 0$, then $\binom{m}{j} \equiv 0$ for any $j \in J_\mu$ since $m < j$.
    Hence $\mu$ belongs to the right-hand set.

    Suppose that $\mu \in \GAMMA(m)$ satisfies $\mu_1,\mu_2 > 0$.
    Then 
    \begin{align}
        \psi_\mu(y) = \binom{m}{0}y^m + \cdots + \binom{m}{m - \mu_2}x^{m - \mu_2}y^{\mu_2} + \cdots + \binom{m}{\mu_1 - 1}x^{\mu_1 - 1}y^{m - \mu_1 + 1}
    \end{align}
    can be divided by $y^{\mu_2}$, that is, $\psi_\mu(y)$ has no $x^jy^{m-j}$ term for $j \in J_\mu$.
    Hence $\binom{m}{j}$ must be zero $j \in J_\mu$.

    Conversely, if $\mu \not\in \GAMMA(m)$, then $\psi_\mu(y)$ cannot be divided by $y^{\mu_2}$.
    This implies that $\psi_\mu(y)$ has a term of $x^{j_0}y^{m-j_0}$ for some $j_0 \in J_\mu$.
    Hence $\binom{m}{j_0} \not\equiv 0 \pmod{p}$.
\end{proof}

\begin{corollary}\label{leq implies GAMMA(m)}
    If $\mu \in \LAMBDA(m)$ satisfies $\mu_1 + \mu_2 \leq m + 1$, then $\mu \in \GAMMA(m)$.
\end{corollary}
\begin{proof}
    This follows from the fact that $J_\mu = \emptyset$ if $\mu_1 + \mu_2 \leq  m+1$.
\end{proof}

\begin{lemma}\label{l412}
    Suppose that $\mu, \nu \in \LAMBDA(m)$ are adjacent and $\mu \subseteq \nu$.
    If $\mu \in \GAMMA(m)$ and $\nu \not\in \GAMMA(m)$, then $\theta_\nu = \alpha\psi_\mu$, where
        \begin{align}
            \alpha = \begin{cases*}
                x & if $\nu_1 = \mu_1 + 1$;\\
                y & if $\nu_2 = \mu_2 + 1$.
            \end{cases*} \label{alpha}
        \end{align}
        Hence \begin{align}
            \exp(\A,\nu) = (m+1,\, \mu_1 + \mu_2) = (m+1,\, \nu_1 + \nu_2 - 1)
        \end{align} and $\DELTA(\nu) < \DELTA(\mu)$.
\end{lemma}
\begin{proof}
    Note that $\deg{\psi_\nu} < \deg{\psi_\nu'}$ since $\nu_1 + \nu_2 > m+1 > m$.
    It is clear that $\theta_\nu \neq \psi_\nu$ by \Cref{GAMMA_p(m)}.
    It follows from \Cref{AN1} that $\theta_\nu = \alpha\theta_\mu = \alpha\psi_\mu$, where $\alpha$ is defined by \eqref{alpha}.
    Hence 
    \begin{align}
        \exp(\A,\nu) = (\deg{\theta_\nu},\, |\nu| - \deg{\theta_\nu}) = (m+1,\, \mu_1 + \mu_2).
    \end{align}
    Furthermore, we have
    \begin{align}
        \DELTA(\nu) = \mu_1 + \mu_2 - (m+1) < \mu_1 + \mu_2 - m = \DELTA(\mu).
    \end{align}
\end{proof}

\begin{theorem}\label{GAMMA_p(m)+}
    For any $m \in \mathbb{Z}_{> 0}$, we have
    \begin{align}
        \GAMMA(m) = \Set{ \mu \in \LAMBDA(m) | \exp(\A,\mu) = (m,\, \mu_1 + \mu_2)}.
    \end{align}
\end{theorem}
\begin{proof}
    Suppose that $\mu \in \GAMMA(m)$.
    By definition, $\{\psi_\mu,\psi_\mu'\}$ is a basis for $D(\A,\mu)$.
    Then $\exp(\A,\mu) = (m,\, \mu_1 + \mu_2)$.

    Suppose that $\mu \not\in \GAMMA(m)$.
    Then $m < \mu_1 + \mu_2$ by \Cref{GAMMA_p(m)}.
    We can take $\nu \in \LAMBDA(m) \setminus \GAMMA(m)$ adjacent to an element of $\GAMMA(m)$ that satisfies $\nu \subseteq \mu$.   
    By \Cref{l412}, the minimum degree of elements of $D(\A,\nu)$ is $m+1$.
    Since $D(\A,\mu) \subseteq D(\A,\nu)$, any element of degree $m$ does not belong to $D(\A,\mu)$.
    Hence $\exp(\A,\mu) \neq (m,\, \mu_1 + \mu_2)$.
\end{proof}

We will now consider to explicitly represent the set $\GAMMA(m)$.
A non-empty subset $\mathcal{L}$ of $\LAMBDA(m)$ is a \textbf{lower set} of $\LAMBDA(m)$ if $\mu \subseteq \nu$ for $\mu \in \LAMBDA(m)$ and $\nu \in \mathcal{L}$ implies $\mu \in  \mathcal{L}$.
Let $\mathcal{L}(S)$ denote the lower set of $\LAMBDA(m)$ generated by $S \subseteq \LAMBDA(m)$.

\begin{proposition}\label{Gamma is an order ideal}
	$\GAMMA(m)$ is a lower set of $\LAMBDA(m)$.
\end{proposition}
\begin{proof}
    Let $\mu \in \LAMBDA(m)$ and $\nu \in \GAMMA(m)$ satisfy $\mu \subseteq \nu$.
    Since $\mu_1 \leq \nu_1$ and $\mu_2 \leq \nu_2$, we have $J_\mu \subseteq J_\nu$.
    Suppose that $j \in J_\mu$.
    Then $j \in J_\nu$.
    Since $\nu \in \GAMMA(m)$, it follows from \Cref{GAMMA_p(m)} that $\binom{m}{j} \equiv 0 \pmod{p}$.
    Hence $\mu \in \GAMMA(m)$ by \Cref{GAMMA_p(m)}.
\end{proof}

\color{black}

A non-empty subset $\mathcal{U}$ of $\LAMBDA(m)$ is an \textbf{upper set} of $\LAMBDA(m)$ if $\mu \supseteq \nu$ for $\mu \in \LAMBDA(m)$ and $\nu \in \mathcal{U}$ implies $\mu \in  \mathcal{U}$.
Let $\mathcal{U}(B)$ denote the upper set of $\LAMBDA(m)$ generated by $B \subseteq \LAMBDA(m)$.

The notion dual to an upper set is a lower set.
In other words, the complement of a lower set is an upper set.
Hence $\LAMBDA(m) \setminus \GAMMA(m)$ is an upper set since $\GAMMA(m)$ is a lower set.

Let
\begin{align}
    B(m) \coloneq \Set{ \mu \in \LAMBDA(m) | \mu_1 + \mu_2 = m + 2,\ \ \binom{m}{\mu_1 - 1} \not\equiv 0 \pmod{p} }.
\end{align}

\begin{theorem}\label{minimality of B}
    $B(m)$ is the set of minimal elements of $\LAMBDA(m) \setminus \GAMMA(m)$.
    Hence $\LAMBDA(m) \setminus \GAMMA(m)$ is an upper set of $\LAMBDA(m)$ generated by $B(m)$:
    \begin{align}
        \LAMBDA(m) \setminus \GAMMA(m) = \mathcal{U}(B(m)).
    \end{align}
\end{theorem}
\begin{proof}
    Let $\mu \in \LAMBDA(m) \setminus \GAMMA(m)$ be a minimal element.
    Then $\mu_1$, $\mu_2 > 0$ since $\binom{m}{-1} = \binom{m}{m+1} = 0$.
    \Cref{leq implies GAMMA(m)} implies that $\mu_1 + \mu_2 \geq m+2$.
    Define
    \begin{align}
        \mu' \coloneqq (\mu_1-1,\, \mu_2,\, m),\quad \mu'' \coloneqq (\mu_1,\, \mu_2-1,\, m).
    \end{align}
    Then $\mu'$ and $\mu''$ belong to $\GAMMA(m)$ by the minimality of $\mu$.
    Assume that $\mu_1 + \mu_2 > m+2$.
    Then $J_\mu = J_{\mu'} \cup J_{\mu''}$, and hence we have $\mu \in \GAMMA(m)$ since $\mu',\mu' \in \GAMMA(m)$.
    This contradicts the fact that $\mu \not\in \GAMMA(m)$.
    Therefore $\mu_1 + \mu_2 = m +2$.
    Then $m - \mu_2 = \mu_1 - 2  < \mu_1 - 1 < \mu_1$, that is, $J_\mu = \{\mu_1 - 1\}$.
    Since $\mu \not\in \GAMMA(m)$, we have $\binom{m}{\mu_1 - 1} \not\equiv 0 \pmod{p}$.
    Hence $\mu \in B(m)$.

    Conversely, suppose that $\mu \in B(m)$.
    Then $J_\mu = \{\mu_1 - 1\}$ since $\mu_1 + \mu_2 = m+ 2$.
    Since $\binom{m}{\mu_1-1} \not\equiv 0 \pmod{p}$, we have $\mu \not\in \GAMMA(m)$.
    The minimality of $\mu \in \LAMBDA(m) \setminus \GAMMA(m)$ follows from \Cref{leq implies GAMMA(m)}.
\end{proof}

For $m \in \mathbb{Z}_{>0}$, define 
\begin{align}
    G_m \coloneqq \Set{ g \in \mathbb{Z}_{\geq 0} | \text{$c_e(g) \leq c_e(m)$ for any $e \in \mathbb{Z}_{\geq 0}$} },
\end{align}
and denote $G_m = \{g_0,g_1,\ldots,g_t\}$, where $0 = g_0 \leq  g_1 \leq \cdots \leq g_t = m$.
If $g \in G_m$, then $m-g \in G_m$ since $c_e(m-g) = c_e(m) - c_e(g)$ for any $e \in \mathbb{Z}_{\geq 0}$.
Hence we have
\begin{align}
    g_i +  g_{t-i} = m \label{g+g=m}
\end{align}
for any $i \in \{0,1,\ldots,t\}$.

\begin{proposition}\label{Bgg}
    \begin{align}
        B(m) = \Set{ (g_i+1,\, g_{t-i}+1,\, m) \in \LAMBDA(m) | i \in \{0,1,\ldots,t\} }.
    \end{align}
\end{proposition}
\begin{proof}
    Let $\mu \in B(m)$.
    Then $\mu_1 - 1 \in G_m$ since $\binom{m}{\mu_1 - 1} \not\equiv 0 \pmod{p}$, that is, there exists $i \in \{0,1,\ldots,t\}$ such that $\mu_1 =  g_i + 1$.
    Moreover, since $\mu_1 + \mu_2 = m+2$, it follows from \eqref{g+g=m} that 
    \begin{align}
        \mu_2 = m + 2 - \mu_1 = m - g_i + 1 = g_{t-i} + 1.
    \end{align}
    
    Let $\mu = (g_i+1,\, g_{t-i}+1,\, m)$ for $i \in \{0,1,\ldots,t\}$.
    Then $\mu_1 + \mu_2 = g_i + g_{t-i} + 2 = m+2$ by \eqref{g+g=m}.
    Since $\mu_1 - 1 = g_i$ satisfies $c_e(\mu_1 - 1) \leq c_e(m)$ for any $e \in \mathbb{Z}_{\geq 0}$, we have $\binom{m}{\mu_1 - 1} \not\equiv 0 \pmod{p}$.
    Hence $\mu \in B(m)$.
\end{proof}

Let
\begin{align}
    S(m) \coloneqq \Set{ (g_i,\, g_{t-i+1},\, m) \in \LAMBDA(m) | i \in \{1,\ldots,t\} }.
\end{align}

\begin{theorem}\label{GAMMAIS(m)}

    $S(m)$ is the set of maximal elements of $\GAMMA(m)$.

\end{theorem}
\begin{proof}
    Let $\mu = (g_i,\, g_{t-i+1},\, m) \in S(m)$ for $i \in \{1,\ldots,t\}$.
    Define 
    \begin{align}
        \mu' &\coloneqq (g_i + 1,\, g_{t-i+1},\, m),& \mu'' &\coloneqq (g_i,\, g_{t-i+1} + 1,\, m),\\
        \nu' &\coloneqq (g_i + 1,\, g_{t-i} + 1,\, m),& \nu'' &\coloneqq (g_{i-1} + 1,\, g_{t-i+1} + 1,\, m).
    \end{align}
    Then $\nu' \subseteq \mu'$ and $\nu'' \subseteq \mu''$ since $g_{i'-1} + 1 \leq g_{i'}$ for each $i' \in \{1,\ldots,t\}$.
    Since $\nu',\nu'' \in B(m)$, it follows from \Cref{minimality of B} that $\mu',\mu'' \in \LAMBDA(m) \setminus \GAMMA(m)$.
    Therefore this implies that $\mu$ is a maximal element of $\GAMMA(m)$.

    Conversely, suppose that $\mu \in \GAMMA(m)$ is maximal.
    Since $\mu' \coloneqq (\mu_1 + 1,\, \mu_2,\, m)$ does not belong to $\GAMMA(m)$, there exists $i \in \{0,1,\ldots,t\}$ such that $(g_i + 1,\, g_{t-i} + 1,\, m) \subseteq \mu'$ by \Cref{minimality of B}.
    Thus we have $g_i \leq \mu_1$.
    Assume that $g_i < \mu_1$.
    Then $g_i + 1 \leq \mu_1$ holds, and hence $(g_i + 1,\, g_{t-i} + 1,\, m) \subseteq \mu$.
    This implies that $\mu \not\in \GAMMA(m)$, which is a contradiction.
    Thus $\mu_1 = g_i$.
    Since $\mu$ and $(g_i,\, g_{t-i+1},\, m) \in S(m)$ are maximal elements of $\GAMMA(m)$, and they are comparable, $\mu$ must be equal to $(g_i,\, g_{t-i+1},\, m)$.
    Hence $\mu \in S(m)$.
\end{proof}

\begin{corollary}\label{main thm:gammam}
    \begin{align}
        \GAMMA(m) = \mathcal{L}(S(m)) \cup \Set{\mu \in \LAMBDA(m) | \text{$\mu_1 = 0$ or $\mu_2 = 0$} }.
    \end{align}
\end{corollary}

The generators of $\LAMBDA(m) \setminus \GAMMA(m)$ can also be characterized as follows.

\begin{corollary}
    \begin{align}
        B(m) = \Set{ \mu \in \LAMBDA(m) | \mu_1 + \mu_2 = m + 2,\ \ \DELTA(\mu) = 0}. \label{BDELTA = 0}
    \end{align}
\end{corollary}
\begin{proof}
    Let $\mu \in B(m)$. Since $\mu$ is a minimal element of $\LAMBDA(m) \setminus \GAMMA(m)$ by \Cref{minimality of B}, we have $\exp(\A,\mu) = (m+1,\, \mu_1 + \mu_2) = (m + 1,\, m+1)$ by \Cref{l412}.
    Hence $\DELTA(\mu) = 0$.

    Conversely, suppose that $\mu \in \LAMBDA(m)$ belongs to the right-hand side of \eqref{BDELTA = 0}.
    Since $|\mu| = 2m+2$, we have $\exp(\A,\mu) = (m+1,\, m+1)$.
    Therefore $\mu \not\in \GAMMA(m)$.
    The minimality of $\mu \in \LAMBDA(m) \setminus \GAMMA(m)$ follows from \Cref{leq implies GAMMA(m)}.
\end{proof}

\begin{corollary}
    For $m \in \mathbb{Z}_{>0}$ and $j \in \{0,\ldots,m\}$, the following are equivalent:
    \begin{itemize}
        \item $\binom{m}{j} \not\equiv 0 \pmod{p}$;
        \item $\DELTA(j+1,\, m+1-j,\, m) = 0$.
    \end{itemize}
\end{corollary}

\begin{example}
    Let $m = 16$ and $p = 3$.
    Then $m = 1 \cdot 3^2 + 0 \cdot 3^1 + 1 \cdot 3^2$, and
    \begin{align}
        G_m = \Set{ 0,1,3,4,6,7,9,10,12,13,15,16 }.
    \end{align}
    Hence 
    \begin{align}
        B(m) &= \left\{ \begin{lgathered}
            (1,17,16),\ (2,16,16),\ (4,14,16),\ (5,13,16),\ (7,11,16),\ (8,10,16),\ \\
            (17,1,16),\ (16,2,16),\ (14,4,16),\ (13,5,16),\ (11,7,16),\ (10,8,16)
        \end{lgathered}\right\},\\
        S(m) &= \left\{ \begin{lgathered}
            (1,16,16),\ (3,15,16),\ (4,13,16),\ (6,12,16),\ (7,10,16),\ (9,9,16),\,\\
            (16,1,16),\ (15,3,16),\ (13,4,16),\ (12,6,16),\ (10,7,16)
        \end{lgathered}\right\}.
    \end{align}
    \Cref{fig:m=16p=3} shows the lowest degree of $D(\A,\mu)$ for each $\mu \in \LAMBDA(m)$, where the shaded area denotes $\GAMMA(m)$ and the brackets indicate the centers of connected components. 

    \begin{table}
    \centering
    \caption{The lower degree of $D(\A,\mu)$ for $\mu \in \LAMBDA(m)$ when $m=16$ and $p = 3$}
    
    \includegraphics[width=\linewidth]{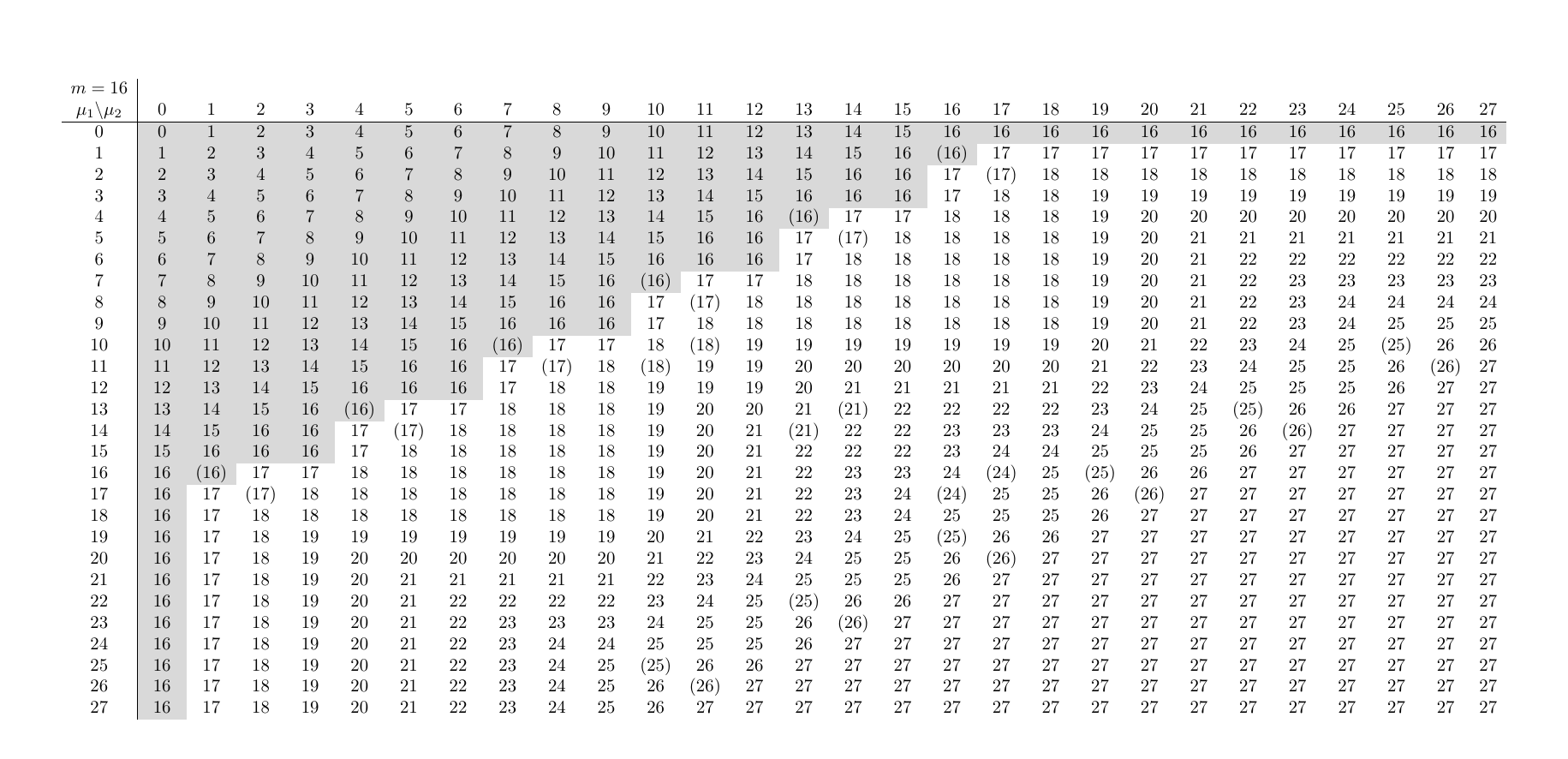}
    
    \label{fig:m=16p=3}
\end{table}

\end{example}

\begin{lemma} \label{gg'}
    \begin{align}
        S(m) = \Set{ (g,g',m) \in \LAMBDA(m) | g,g' \in G_m,\ \ g + g' = p^{s(g)} + \sum_{e \geq s(g)}c_e(m)p^e }.
    \end{align}
    Furthermore, if $(g,g',m) \in S(m)$, then $s(g) = s(g')$.
\end{lemma}
\begin{proof}
    Let $(g_i,g_{t-i+1},m) \in S(m)$.
    It can be seen that two adjacent numbers $g_{i-1}, g_i \in G_m$ have the relation
\begin{align}
    g_i = g_{i-1} + p^{d_{i-1}} - \sum_{0 \leq e < d_{i-1}}c_e(m)p^e, \label{gadj}
\end{align}
where $d_{i-1} \coloneqq \min\Set{e \in \mathbb{Z}_{>0} | c_e(g_{i-1}) < c_e(m) }$.
Furthermore, we have $s(g_i) = d_{i-1}$ by the above equation.
    It follows from \eqref{g+g=m} and \eqref{gadj} that
\begin{align}
    \begin{split}
        g_i + g_{t-i+1} &= m + p^{s(g_i)} - \sum_{0 \leq e < s(g_i)}c_e(m)p^e\\
        &= p^{s(g_i)} + \sum_{e \geq s(g_i)}c_e(m)p^e. \label{g+gS(m)}
    \end{split}
\end{align}
Furthermore, we have $s(g_i) = s(g_{t-i+1})$ by the above equations.

Conversely, suppose that $g,g' \in G_m$ satisfy
\begin{align}
    g + g' = p^{s(g)} + \sum_{e \geq s(g)}c_e(m)p^e.
\end{align}
We can assume that $g = g_i$ for $i \in \{1,\ldots,t\}$.
Then the equation \eqref{g+gS(m)} implies that $g' = g_{t-i+1}$.
Hence $(g,g',m)$ belongs to $S(m)$.
\end{proof}

\begin{lemma}\label{gg'v}
    Let $m \in \mathbb{Z}_{>0}$ and $g,g' \in G_m$ satisfy $(g,g',m) \in S(m)$.
    Then 
    \begin{align}
        \Set{ n \in \mathbb{Z}_{>0} | (g,g',n) \in S(n) } &= \Set{n \in \mathbb{Z}_{>0} | \text{$c_e(n) = c_e(m)$ for any $e \geq s(g)$} }\\
        &= \Set{ g + g' - v | v \in \{1,\ldots,p^{s(g)}\} }.
    \end{align}
\end{lemma}
\begin{proof}
    Suppose that $(g,g',n) \in S(n)$.
    Since $(g,g',m) \in S(m)$, by \Cref{gg'},
    \begin{align}
        g + g' = p^{s(g)} + \sum_{e \geq s(g)}c_e(m)p^e = p^{s(g)} + \sum_{e \geq s(g)}c_e(n)p^e \label{mandn}
    \end{align}
    holds, and hence $c_e(n) = c_e(m)$ for any $e \geq s(g)$.

    Let $n \in \mathbb{Z}_{>0}$ satisfy $c_e(n) = c_e(m)$ for any $e \geq s(g)$.
    Then $g,g' \in G_n$ since $c_e(g)$, $c_e(g') \leq c_e(m) = c_e(n)$.
    Moreover, since $g,g'$ satisfy \eqref{mandn}, \Cref{gg'} implies that $(g,g',n) \in S(n)$.
\end{proof}

Let $\mu \in \LAMBDA$ and $C$ be the connected component containing $\mu$.
For $t \in \{1,2,3\}$, define $C_{\mu,t}$ to be the set of $\nu \in C$ satisfying the following two conditions (See {\cite[Definition 2.5]{AN}):
\begin{itemize}
    \item $\nu_{t+1} = \mu_{t+1}$ and $\nu_{t+2} = \mu_{t+2}$;
    \item If $\nu \subseteq \kappa \subseteq \mu$ or $\mu \subseteq \kappa \subseteq \nu$, then $\kappa \in C$.
\end{itemize}
The set $C_{\mu,t}$ has the unique element $\kappa \in C_{\mu,t}$ such that $\DELTA(\kappa) \geq \DELTA(\nu)$ for any $\nu \in C_{\mu,t}$.
In other words, the restriction $\DELTA|_{C_{\mu,t}}$ is unimordal.
We call this element $\kappa$ the \textit{peak element} of $C_{\mu,t}$ (See {\cite[Lemma 4.5 and Definition 4.6]{AN}).

\begin{lemma}\label{l416}
    Every $\kappa = (g,g',m) \in S(m)$ is the peak element of $C_{\kappa,t}$ for any $t \in \{1,2\}$.
\end{lemma}
\begin{proof}
    To prove the statement, since $\DELTA|_{C_{\kappa,t}}$ is unimordal, we only show that $\DELTA(\mu) < \DELTA(\kappa)$ for any $\nu \in C_{\kappa,t}$ adjacent to $\kappa \in S(m)$. 

    If $\mu \subseteq \kappa$, then $\nu \in \GAMMA(m)$.
    Since $g + g' \geq m+1$ by \eqref{g+gS(m)}, it satisfies $\mu_1 + \mu_2 \geq m$.
    Hence \Cref{GAMMA_p(m)+} implies that 
    \begin{align}
        \DELTA(\mu) = |\mu| - 2m  < |\kappa| - 2m = \DELTA(\kappa).
    \end{align}
    Conversely, if $\mu \supseteq \kappa$, then $\mu \not\in \GAMMA(m)$.
    Hence \Cref{l412} implies that $\DELTA(\mu) < \DELTA(\kappa)$.
    Thus $\kappa$ is the peak element of $C_{\kappa,t}$.
\end{proof}

\begin{theorem}
    If $\kappa = (g,g',m) \in S(m)$ satisfies $m = g + g' - p^{s(g)}$, then $\kappa \in Z(p^{s(g)})$.
\end{theorem}
\begin{proof}
    To prove that $\kappa \in Z(p^{s(g)})$, we show that $\kappa$ is the peak element of $C_{\kappa,3}$.
    Let $\nu \coloneqq (g,\, g',\, m+1) \in C_{\kappa,3}$.
    Then $\nu \in S(m+1)$ by \Cref{gg'v}.
    Therefore it follows from \Cref{GAMMAIS(m)} and \Cref{GAMMA_p(m)} that $\exp(\A,\nu) = (m+1,\, g+g')$, and hence
    \begin{align}
        \DELTA(\nu) = g + g' - (m+1) < g + g' - m = \DELTA(\kappa).
    \end{align}
    Let $\mu \coloneqq (g,\,g',\,m-1) \in C_{\kappa,3}$.
    Since $g + g' = m + p^{s(g)}$, \Cref{gg'} implies that $s(m) = s(g)$, and hence
    \begin{align}
        m - 1 = \sum_{e > s(g)}c_e(m)p^e + \Bigl(c_{s(g)}(m)-1\Bigr)p^{s(g)} + \sum_{0 \leq e < s(g)}(p-1)p^e.
    \end{align}
    Define
    \begin{align}
        h &\coloneqq \sum_{e > s(g)}c_e(g) + \Bigl( c_{s(g)}(g) - 1 \Bigr)p^{s(g)} + \sum_{0 \leq e < s(g)}(p-1)p^e,\\
        h' &\coloneqq \sum_{e > s(g)}c_e(g') + \Bigl( c_{s(g)}(g') - 1 \Bigr)p^{s(g)}.
    \end{align}
    Then $h < g$ and $h' < g'$, that is, $(h+1,\, h'+1,\, m-1) \subseteq \mu$.
    Moreover, $h$ and $h'$ belong to $G_{m-1}$, and $h + h' = m-1$.
    Hence $(h+1,\, h'+1,\, m-1) \in B(m-1)$ by \eqref{g+g=m} and \Cref{Bgg}.
    This implies that $\mu \not\in \GAMMA(m-1)$.
    Then $\deg{\theta_\mu} > m-1$ by \Cref{GAMMA_p(m)+}, and hence
    \begin{align}
        \DELTA(\mu) = |\mu| - 2\deg{\theta_\mu} < |\kappa|-2m = \DELTA(\kappa).
    \end{align}
    Therefore $\kappa$ is the peak element of $C_{\kappa,3}$, and \Cref{l416} implies that $\kappa \in Z$.
    Since $\DELTA(\kappa) = |\kappa| - 2m = p^{s(g)}$, we have $\kappa \in Z(p^{s(g)})$.
\end{proof}

The centers obtained above are those of the connected components of $\LAMBDA^\bal_{>0}$ closest to the unbalanced area $\Set{\mu \in \LAMBDA | \mu_1 + \mu_2 < \mu_3}$.
As in exponents, it is sufficient to study bases only for the centers of the connected components, but it would be difficult to explicitly represent the bases of connected components other than those obtained above.

\begin{example}\label{eg:334_ch23}
    Let $\mu = (3,3,4)$.
    We compute a basis and exponents for $D(\A,\mu)$ when $\Char{\mathbb{F}} = 2$ and $3$.
    See also \Cref{eg334} for the case $\Char{\mathbb{F}} = 0$.

    First, suppose that $\Char{\mathbb{F}} = 2$.
    Then $G_4 = \{0,4\}$, and hence $S(4) = \{(4,4,4)\}$.
    Since $\mu \subseteq (4,4,4)$, we have $\mu \in \GAMMA(4)$.
    Therefore 
    \begin{align}
        \{\psi_\mu,\psi_\mu'\} = \left\{ x^4\partial_x + y^4\partial_y,\ x^3y^3(\partial_y - \partial_x) \right\}
    \end{align}
    is a basis for $D(\A,\mu)$ and $\exp(\A,\mu) = (4,6)$.
    
    Second, suppose that $\Char{\mathbb{F}} = 3$.
    Then $G_4 = \{0,1,3,4\}$, and hence $\mu = (3,3,4) \in S(4)$.
    We have $\mu \in \GAMMA(4)$.
    Therefore 
    \begin{align}
        \{\psi_\mu,\psi_\mu'\} = \left\{ (x^4+x^3y)\partial_x + (xy^3 + y^4)\partial_y,\ x^3y^3(\partial_y-\partial_x) \right\}
    \end{align}
    is a basis for $D(\A,\mu)$ and $\exp(\A,\mu) = (4,6)$.
\end{example}

\begin{example}
    Let $p = 3$ and $\mu = (41,52,31)$.
    In \Cref{eg31}, we checked $\exp(\A,\mu) = (58,66)$.
    Define $\nu \coloneqq \mu - (27,27,0) = (14,25,31)$.
    By \Cref{Period2}, we have $\exp(\A,\nu) = \exp(\A,\mu) - (27,27) = (31,39)$, and \Cref{GAMMA_p(m)+} implies that $\nu \in \GAMMA(31)$.
    Hence a basis $\{\theta_\nu,\theta_\nu'\}$ for $D(\A,\nu)$ can be computed as 
    \begin{align}
        \theta_\nu = \psi_\nu &= \sum_{j=14}^{31}\binom{31}{j}x^jy^{31-j}\partial_x + \sum_{j=0}^{13}\binom{31}{j}x^jy^{31-j}\partial_y\\
        &= (x^{31} + x^{30}y  + x^{28}y^3 + x^{27}y^4)\partial_x + (x^4y^{27} + x^3y^{28} + xy^{30} + y^{31})\partial_y,\\[6pt]
        \theta_\nu' = \psi_\nu' &= x^{14}y^{25}(\partial_y - \partial_x).
    \end{align}
    By \Cref{Period2}, we can compute a basis $\{\theta_\mu,\theta_\mu'\}$ for $D(\A,\mu)$ as
    \begin{align}
        \theta_\mu = \PSI_3(\theta_\nu) &= (x^{58} + x^{57}y  + x^{55}y^3 + x^{54}y^4)\partial_x - (x^4y^{54} + x^3y^{55} + xy^{57} + y^{58})\partial_y,\\[6pt]
        \theta_\mu' = \PSI_3(\theta_\nu') &= - x^{41}y^{25}\partial_x - x^{14}y^{52}\partial_y.
    \end{align}
\end{example}

Let $m,d \in \mathbb{Z}_{\geq 0}$ satisfy $m \leq p^d$ and let $\mu \in \LAMBDA_{\leq p^d} \cap \LAMBDA(m)$.
Recall that 
\begin{align}
    \mu^\vee = (p^d - \mu_2,\, p^d-\mu_1,\, m).
\end{align}
Define 
\begin{align}
    \GAMMA^\vee(m) \coloneqq \Set{\mu \in \LAMBDA_{\leq p^d} | \mu^\vee \in \GAMMA(m) }.
\end{align}
Then 
\begin{align}
    (\psi_\mu)^\vee &= \sum_{j=\mu_1}^m\binom{m}{j}x^{p^d-\mu_2+m-j}y^{j-\mu_1}\partial_x - \sum_{j=0}^{\mu_1-1}\binom{m}{j}x^{m-j-\mu_2}y^{p^d  - \mu_1 + j}\partial_y,\\ 
    (\psi_\mu')^\vee &= x^{p^d}\partial_x + y^{p^d}\partial_y,
\end{align}
and \Cref{veebasis} implies the following.

\begin{proposition}
\begin{align}
    \GAMMA^\vee(m) = \Set{\mu \in \LAMBDA(m) | \text{$\{(\psi_\mu)^\vee,(\psi_\mu')^\vee\}$ is a basis for $D(\A,\mu^\vee)$}}.
\end{align}
\end{proposition}

\section*{Acknowledgment}
The authors wish to thank Professor Yasuhide Numata for his valuable advice and providing a program to efficiently compute bases for the modules of logarithmic vector fields based on an algorithm in \cite{numata2007algorithm}. 
S. T. was supported by JSPS KAKENHI, Grant Number JP22K13885.

\bibliographystyle{amsplain}
\bibliography{bibfile}

\providecommand{\bysame}{\leavevmode\hbox to3em{\hrulefill}\thinspace}
\providecommand{\MR}{\relax\ifhmode\unskip\space\fi MR }
\providecommand{\MRhref}[2]{%
  \href{http://www.ams.org/mathscinet-getitem?mr=#1}{#2}
}
\providecommand{\href}[2]{#2}
\begin{thebibliography}{10}

\bibitem{AA}
T.~Abe, \emph{Chambers of 2-affine arrangements and freeness of
  3-arrangements}, Journal of Algebraic Combinatorics \textbf{38} (2013),
  65--78.

\bibitem{AN}
T.~Abe and Y.~Numata, \emph{Exponents of 2-multiarrangements and multiplicity
  lattices}, Journal of Algebraic Combinatorics \textbf{35} (2012), no.~1,
  1--17.

\bibitem{abe2013free-mz}
T.~Abe and M.~Yoshinaga, \emph{Free arrangements and coefficients of
  characteristic polynomials}, Mathematische Zeitschrift \textbf{275} (2013),
  no.~3-4, 911--919.

\bibitem{FWY}
M.~Feigin, Z.~Wang, and M.~Yoshinaga, \emph{Integral expressions for
  derivations of multiarrangements}, International Journal of Mathematics
  \textbf{36} (2025), no.~04, 2450085, Publisher: World Scientific Publishing
  Co.

\bibitem{numata2007algorithm}
Y.~Numata, \emph{An {Algorithm} to {Construct} {A} {Basis} for the {Module} of
  {Logarithmic} {Vector} {Fields}}, June 2007, arXiv:0707.0004 [math].

\bibitem{saito1980theory-jotfosuotsim}
K.~Saito, \emph{Theory of logarithmic differential forms and logarithmic vector
  fields}, Journal of the Faculty of Science. University of Tokyo. Section IA.
  Mathematics \textbf{27} (1980), no.~2, 265--291.

\bibitem{terao1983exponents-spc}
H.~Terao, \emph{The exponents of a free hypersurface}, Singularities, {Part} 2
  ({Arcata}, {Calif}., 1981), Proc. {Sympos}. {Pure} {Math}., vol.~40, Amer.
  Math. Soc., Providence, RI, 1983, pp.~561--566.

\bibitem{W}
A.~Wakamiko, \emph{On the {Exponents} of 2-{Multiarrangements}}, Tokyo Journal
  of Mathematics \textbf{30} (2007), no.~1, 99--116.

\bibitem{yoshinaga2004characterization-im}
M.~Yoshinaga, \emph{Characterization of a free arrangement and conjecture of
  {Edelman} and {Reiner}}, Inventiones mathematicae \textbf{157} (2004), no.~2,
  449--454.

\bibitem{yoshinaga2005freeness-botlms}
\bysame, \emph{On the freeness of 3-arrangements}, Bulletin of the London
  Mathematical Society \textbf{37} (2005), no.~01, 126--134.

\bibitem{Zi}
G.~M. Ziegler, \emph{Multiarrangements of hyperplanes and their freeness},
  Contemporary Mathematics \textbf{90} (1989), 345--359.

\end{thebibliography}

\end{document}